\documentclass[final,3p,times]{elsarticle}

\usepackage{amssymb, amsmath, amsfonts}
\usepackage{booktabs}
\usepackage{amsthm}
\usepackage{tikz}
\usepackage{subfigure}
\usepackage{pgfplots}
\pgfplotsset{compat=newest}
\usetikzlibrary{shapes.geometric,decorations.pathmorphing,patterns, shapes.symbols, arrows.meta}
\usetikzlibrary{calc,angles}
\tikzset{block/.style={draw,rectangle, align=center, inner sep=5pt},
oval/.style={draw,ellipse, align = center, inner sep=5pt},
decision/.style={draw,diamond, aspect = 2, align=center,inner sep=2pt}}
\newcommand\leftAngle[4]{
  \pgfmathanglebetweenpoints{\pgfpointanchor{#2}{center}}{\pgfpointanchor{#3}{center}}
  \coordinate (tmpRA) at ($(#2)+(\pgfmathresult-45:#4)$);
  \draw[black] ($(#2)!(tmpRA)!(#1)$) -- (tmpRA) -- ($(#2)!(tmpRA)!(#3)$);
}

\usepackage{paralist}
\usepackage{natbib}
\usepackage{hyperref}
\hypersetup{
    colorlinks=true,
    linkcolor=blue,
    filecolor=magenta,      
    urlcolor=cyan,
}
\usepackage{cleveref}
\usepackage[font=small,labelfont=bf]{caption}
\usepackage{graphicx}
\graphicspath{{./figures/}}
\usepackage{comment}

\makeatletter
\def\namedlabel#1#2{\begingroup
   \def\@currentlabel{#2}%
   \label{#1}\endgroup
}
\makeatother

\journal{Applied Numerical Mathematics}

\newcommand{\R}{\mathbb R}

\newcommand{\transpose}{\mathsf{T}}

\DeclareMathOperator{\Ls}{L}
\DeclareMathOperator{\LLs}{\mathbf L}
\DeclareMathOperator{\Hs}{H}
\DeclareMathOperator{\HHs}{\mathbf H}
\DeclareMathOperator{\Ws}{W}

\DeclareMathOperator{\Cs}{C}
\DeclareMathOperator{\CCs}{\mathbf C}
\DeclareMathOperator{\Ds}{D}

\DeclareMathOperator{\Vs}{V}

\DeclareMathOperator{\Xs}{X}
\DeclareMathOperator{\Ys}{Y}

\newcommand{\brac}[1]{\left\lbrace{#1}\right\rbrace}

\newcommand{\paren}[1]{\left({#1}\right)}
\newcommand{\norm}[1]{\left\lVert{#1}\right\rVert}

\newcommand{\abs}[1]{\left\vert{#1}\right\vert}

\newcommand{\inprod}[1]{\left\langle{#1}\right\rangle}

\newcommand{\ovl}[1]{\overline{#1}}

\newcommand{\vertiii}[1]{{\left\vert\kern-0.25ex\left\vert\kern-0.25ex\left\vert #1 
    \right\vert\kern-0.25ex\right\vert\kern-0.25ex\right\vert}}

\newcommand{\dive}[1]{\nabla\cdot{#1}}

\newcommand{\veps}{\varepsilon}

\newcommand{\pa}{\partial}

\newcommand{\Gm}{\Gamma}

\newcommand{\Om}{\Omega}

\newcommand{\vphi}{\varphi}
\newcommand{\fa}{\forall}
\newcommand{\sst}{\subset}

\newcommand{\xb}{\boldsymbol{x}}

\newcommand{\vb}{\mathbf{v}}

\newcommand{\nv}{\textbf{n}}

\newcommand{\dx}{\,\mathrm{d}\xb}
\newcommand{\ds}{\,\mathrm{d}s}
\newcommand{\dt}{\,\mathrm{d}t}
\newcommand{\di}{\,\mathrm{d}}

\newcommand{\q}{\quad}
\newcommand{\qq}{\qquad}
\newcommand{\qqq}{\qquad\quad}
\newcommand{\qqqq}{\qquad\qquad}
\newcommand{\qqqqq}{\qquad\qquad\quad}

\newtheorem{theorem}{Theorem}[section]
\newtheorem{lemma}{Lemma}[section]

\newtheorem{corollary}{Corollary}[section]

\newtheorem{remark}{Remark}[section]

\begin{document}

\begin{frontmatter}



\title{A fitted space-time finite element method for an advection-diffusion problem with moving interfaces}


\author[1]{Quang Huy Nguyen}
\ead{huy.nguyenquang1@hust.edu.vn}

\author[2]{Van Chien Le}
\ead{vanchien.le@ugent.be}

\author[1]{Phuong Cuc Hoang}
\ead{cuc.hp222163m@sis.hust.edu.vn}

\author[1]{Thi Thanh Mai Ta\corref{cor}}
\ead{mai.tathithanh@hust.edu.vn}

\address[1]{Faculty of Mathematics and Informatics, Hanoi University of Science and Technology, 11657 Hanoi, Vietnam.}

\address[2]{IDLab, Department of Information Technology, Ghent University - imec, 9000 Ghent, Belgium.} 

\cortext[cor]{Corresponding author}

\begin{abstract}
This paper presents a space-time interface-fitted finite element method for solving a parabolic advection-diffusion problem with a nonstationary interface. The jumping diffusion coefficient gives rise to the discontinuity of the solution gradient across the interface. We use the Banach-Ne{\v c}as-Babu\v{s}ka theorem to show the well-posedness of the continuous variational problem. A fully discrete finite-element based scheme is analyzed using the Galerkin method and unstructured interface-fitted meshes. An optimal error estimate is established in a discrete energy norm under a globally low but locally high regularity condition. Some numerical results corroborate our theoretical results.
\end{abstract}




\end{frontmatter}


\section{Introduction}
\label{sec:introduction}

\subsection{Problem statement}
Let $\Om \sst \R^d$, with $d = 1$ or $2$, be an open, bounded domain with Lipschitz continuous boundary $\pa\Om$. The domain $\Om$ consists of two time-dependent subdomains $\Om_1(t)$ and $\Om_2(t)$ sharing a common interface $\Gm(t)$. Mathematically speaking
$$
    \Om = \Om_1(t) \cup \Om_2(t) \cup \Gm(t), \qqqqq \qqq \pa\Om_1(t) \cap \pa\Om_2(t) = \Gm(t),
$$ 
for all $t \in [0, T]$, with $T > 0$. The evolution of the interface over time is determined by a velocity vector $\vb \in \Cs([0, T], \CCs^2(\Omega))$ satisfying $\nabla \cdot \vb(\xb, t) = 0$ for all $(\xb, t) \in \Om \times [0, T]$ \cite{VR2018} (see Figure~\ref{fig: model}). 

\begin{figure}[http]
    \centering
    \begin{tikzpicture}
        \draw (0.5,0) circle (2.5cm);
        \draw[->, line width=0.2mm] (-0.6,0) -- (1,0);
        \node at (0.7, 0.3) {$\vb$};
        \node at (0.2, -0.65) {$\Gamma(t)$};
        \node at (1.8, 1.5) {$\Omega_2(t)$};
        \node at (2.7, -2.2) {$\Omega$};
        \node[oval, fill=gray!30] at (-0.6, 0) {$\Omega_1(t)$};
    \end{tikzpicture}
    \caption{The domain $\Omega$ consisting of two subdomains $\Om_1(t)$ and $\Om_2(t)$ moving with a velocity $\vb$, separated by the interface $\Gamma(t)$.}
    \label{fig: model}
\end{figure}
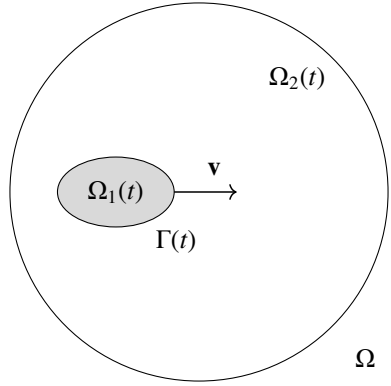

We denote the space-time domain by $Q_T := \Om \times (0, T)$ and its subdomains by
$$
    Q_{i} := \left\{(\xb, t) \mid \xb \in \Omega_{i}(t), \, t \in (0,T)\right\},
$$
with $i = 1, 2$. The two space-time subdomains are separated by the space-time interface 
$$
    \Gamma^{\ast} := \bigcup_{t \in (0,T)} \Gamma(t)\times \left\{t\right\} = \pa Q_1 \cap \pa Q_2.
$$
Throughout the paper, we assume that $\Gamma^\ast$ is a $\Cs^2$-continuous hypersurface in $\mathbb{R}^{d+1}$. 

This paper is concerned with the following problem\footnote{In this paper, we restrict ourselves to the advection-diffusion problem \eqref{eq: IBVP} with non-jumping interface conditions. In general, the interface conditions can be jumping. If this is the case, one can extend the ideas in \cite{CZ1998, ABGL2023} to reduce it to the problem \eqref{eq: IBVP}.}
\begin{equation}
    \label{eq: IBVP}
    \begin{cases}
        \pa_t u + \vb \cdot \nabla u - \dive{\left(\kappa \nabla u\right)} = f \qqq & \text{in} \q Q_T, \\
        \left[u\right] = 0 & \text{on} \q \Gm^\ast, \\
        \left[\kappa \nabla u \cdot \nv\right] = 0 & \text{on} \q \Gm^\ast, \\
        u = 0 & \text{on} \q \pa \Om \times \paren{0, T}, \\
        u (\cdot, 0) = 0 & \text{in} \q \Om,
    \end{cases}
\end{equation}
where $\nv$ is the unit normal to $\Gm(t)$ pointing from $\Om_1(t)$ to $\Om_2(t)$ and $f$ is the given source. The notation $\left[u\right] := \left(u_1\right)_{\mid \Gamma(t)} - \left(u_2\right)_{\mid \Gamma(t)}$ stands for the jump of $u$ across the interface $\Gm(t)$, with $\left(u_i\right)_{\mid \Gamma(t)}$ the limiting value of $u$ on $\Gm(t)$ from $\Om_i(t)\ (i=1,2)$. The first equation in \eqref{eq: IBVP} is usually called an advection-diffusion equation. The second and third conditions describe the behavior of the unknown $u$ and its spatial gradient across the interface $\Gm^\ast$. The problem is supplemented with a homogeneous Dirichlet boundary condition and a homogeneous initial condition. For simplicity, we assume that the diffusion coefficient $\kappa$ is a positive constant on each subdomain, i.e.,
$$
\kappa = 
\begin{cases}
    \kappa_1 > 0 & \text{in} \q \Om_1(t), \\
    \kappa_2 > 0 & \text{in} \q \Om_2(t).
\end{cases}
$$
In practice, the two values $\kappa_1$ and $\kappa_2$ might characterize two different materials occupying two subdomains. Hence, they are generally different, leading to the jump of $\kappa$ across the interface $\Gm(t)$.

\subsection{Relevant literature}

The advection-diffusion problem \eqref{eq: IBVP} arises in various engineering applications of flows and other physical phenomena. The unknown function $u$ may represent the concentration of a pollutant subject to the diffusion effect and is transported by a stream moving at velocity $\vb$. Alternatively, it can represent the electron concentration in an electromagnetic device. In addition, the advection-diffusion equation \eqref{eq: IBVP} can model the eddy-current problem in electromagnetics \cite{LSV2021a}, the heat transfer \cite{Slodicka2021}, or the induction heating process \cite{LSV2024}, which involve moving multiple-component systems.

The difficulties of solving an interface problem are not only the non-smoothness of the solution but also related to the advection component. More specifically, the discontinuity of the diffusion coefficient $\kappa$ at the interface results in a solution with a discontinuous gradient across the interface. Consequently, classical finite element methods applied to this class of problems converge with only sub-optimal orders, e.g., \cite{Babuka1970, CZ1998}. In handling the issue, two major approaches have been extensively investigated over the last few decades: interface-unfitted and interface-fitted methods. The former approach does not require a triangulation that fits the interface. Instead, it accurately approximates discontinuous quantities by adjusting local finite element basis functions on interface elements. Examples of this class include the multiscale finite element method (MsFEM) \cite{CGH2010}, the immersed finite element method (IFEM) \cite{Guo2021}, and the extended finite element method (XFEM) \cite{BB1999}. In contrast, the interface-fitted methods generate a mesh that matches the interface, i.e., the mesh prevents the interface from cutting through an element or resolves the interface approximately to a certain precision \cite{HAC1974, Winslow1966, DLTZ2006}. 

When the interface evolves, we cope with a moving discontinuity. The interface-unfitted approach takes advantage of allowing interface-independent simplicial triangulations, which is convenient when the interface depends on time. In \cite{Zunino2013}, the authors combined the XFEM with a backward Euler time discretization to solve a moving-interface problem. Lehrenfeld and Reusken in \cite{LR2013, Lehrenfeld2015b} proposed a second-order XFEM scheme using a space-time discontinuous Galerkin discretization. More recently, Badia et al. \cite{BDV2023} presented the aggregation finite element method (AgFEM) for tackling the ill-conditioning and small cut cell issues in most interface-unfitted schemes. However, one standing drawback of these interface-unfitted methods is that they require special space-time quadrature rules on cut elements, resulting in high computational costs. In contrast to the success of interface-unfitted schemes, interface-fitted methods have received very little attention in solving moving-interface problems. The reason is that the re-meshing procedure at each time step introduces additional errors in interpolating two consecutive meshes, which can exceed the feasible effort. 

Related to our setting, we also refer to the studies \cite{BMF+2018, LRS2020} that utilized the arbitrary Lagrangian-Eulerian (ALE) technique to solve moving-domain or moving-interface problems. This technique allows for good domain or interface resolution but is less practical when dealing with large deformations. 

\subsection{Contributions and outline}

Our work resolves a challenge of the interface-fitted approach and proposes a new interface-fitted strategy for moving-interface advection-diffusion problems. We introduce a numerical scheme combining a space-time discretization with an interface-fitted mesh continuous Galerkin approximation. In contrast to all mentioned studies invoking the classical time-stepping methods or time-discontinuous Galerkin methods, our approach treats the time variable as another spatial variable. It solves the problem \eqref{eq: IBVP} in the space-time setting. The literature refers to this strategy as the space-time finite element method, introduced by Steinbach in \cite{Steinbach2015} for solving non-interface heat equations. In this way, we can employ fully unstructured space-time meshes and cope with geometrically complicated interfaces. Furthermore, our approach takes advantage of the accuracy and practicality of interface-fitted methods with neither the implementation of costly re-meshing algorithms nor the computation of interaction integrals over moving interfaces or enrichment elements. 

The space-time finite element method has been used in recent studies for solving moving-interface or moving-domain problems \cite{LMN2016, LMS2019, GGS2024}. The authors of these works provided standard error estimates under the assumption that the solution is sufficiently smooth across the interface. This regularity assumption, however, is typically not satisfied in practical situations when the coefficient $\kappa$ is discontinuous. In our work, we instead consider a globally low regularity solution. To rectify the non-smooth property of the solution and derive a priori error analysis for the discretization scheme, we combine the space-time technique \cite{Steinbach2015} with the Stein extension operators \cite{Stein1971}, which allows us to interpolate functions with locally high but globally low regularity. We arrive at an optimal convergence rate with respect to a mesh-dependent energy norm under an appropriate regularity condition. In the case of a weaker assumption, the convergence rate is nearly optimal when the spatial domain is one-dimensional (1D) and sub-optimal for the two-dimensional (2D) spatial domain. To the best of our knowledge, these types of error bounds do not appear anywhere else in the literature on moving-interface problems using the space-time method. 

The paper is divided as follows. The following section provides necessary functional settings, followed by a variational formulation of the problem \eqref{eq: IBVP} in a space-time Sobolev space. The Banach-Ne{\v c}as-Babu\v{s}ka theorem is involved in proving the well-posedness of the continuous variational problem. Section \ref{sec: discretization} is devoted to a Galerkin finite element discretization of this problem. We present the solvability of the discrete problem and derive a priori error estimates in a discrete energy norm under different smoothness assumptions. Some numerical experiments are performed in Section \ref{sec: results} to verify the theoretical analysis. Finally, Section \ref{sec: conclusions} contains some conclusion remarks.

\section{Variational formulation}
\label{sec:var_form}

\subsection{Functional setting}

For $s\ge 0, p\in [1, \infty]$ and a bounded Lipschitz domain $\Omega$, we denote by $\Ws^{s,p}\left(\Omega\right)$ the standard Sobolev space, where $\Ws^{0,p}\left(\Omega\right)$ coincides with the Lebesgue space $\Ls^p\left(\Omega\right)$ \cite[Chapter~2]{Ern2021c}. Among Sobolev spaces, only $\Ws^{s,2}\left(\Omega\right)$ forms a Hilbert space and is usually denoted by $\Hs^s\left(\Omega\right)$. Note that when $\Omega\subset \mathbb{R}^2$, the space $\Hs^1\left(\Omega\right)$ is continuously embedded into $\Ls^{p}\left(\Omega\right)$ for all $p \in [1, \infty)$. When $\Omega\subset \mathbb{R}^3$, this continuous embedding only holds for $p \in [1, 6]$. We also need the compact embedding of $\Hs^s\left(\Omega\right)$ into $\Cs\left(\ovl{\Omega}\right)$ for all $s > \frac{m}{2}$, where $\Omega \subset \mathbb{R}^{m}\ \left(m = 2 \text{ or } 3\right)$ \cite[Section~2.3]{Ern2021c}. 

For $l,k\in \mathbb{N}$ and a space-time domain $Q_T$, we recall from \cite[Section~1.4]{WYW2006} the anisotropic Sobolev space 
$$
\Hs^{l, k}\left(Q_T\right) := \left\{u\in \Ls^2\left(Q_T\right)\mid \partial_{\boldsymbol{x}}^{\boldsymbol{\alpha}} \partial_t^r u \in \Ls^2\left(Q_T\right)\text{ for all } \abs{{\boldsymbol{\alpha}}}\le l,\ 0\le r \le k\right\},
$$
furnished with the graph norm 
$$
    \norm{u}^2_{\Hs^{l, k}\left(Q_T\right)} := \sum_{\abs{{\boldsymbol{\alpha}}}\le l,\ 0\le r \le k} \norm{\partial_{\boldsymbol{x}}^{\boldsymbol{\alpha}} \partial_t^r u}^2_{\Ls^2(Q_T)} \qqqq \forall u \in \Hs^{l, k}\left(Q_T\right).
$$
When $l=k$, we retain the classical Sobolev space $\Hs^k\left(Q_T\right)$. The notation $\Hs^{1,0}_0\left(Q_T\right)$ stands for the closure of $\Cs_0^1\left(Q_T\right)$ with respect to the norm $\norm{\cdot}_{\Hs^{1,0}\left(Q_T\right)}$, where $\Cs_0^1\left(Q_T\right)$ denotes the space of continuously differentiable functions with a compact support in $Q_T$. Since the space $\Hs^{1,0}_0\left(Q_T\right)$ is frequently used below, we shall use the compact notation 
$$
\Ys := \Hs^{1,0}_0\left(Q_T\right),
$$
equipped with the equivalent norm
$$
\norm{u}_{\Ys}^2 := \int\limits_0^T \int\limits_\Om \kappa \abs{\nabla u}^2 \dx \dt\qqqq \forall u \in \Ys.
$$
The dual space of $\Ys$ is denoted by $\Ys^\prime$, and the duality pairing between $\Ys^\prime$ and $\Ys$ is denoted by $\inprod{\cdot, \cdot}_{\Ys^\prime \times \Ys}$.
For $u \in \Ys$, its distributional time derivative $\partial_t u$ is defined by
$$
\partial_t u \left(\psi\right) := - \int\limits_0^T \int\limits_\Om u \partial_t \psi \dx \dt\qqqq \fa \psi \in \Cs_0^1\left(Q_T\right).
$$
Then, we introduce the spaces
$$
\Xs := \brac{u \in \Ys \mid \partial_t u \in \Ys^\prime},\qqqq \Xs_0 := \brac{u \in \Xs \mid u(\cdot, 0) = 0},
$$
endowed with the graph norm 
$$
\norm{u}^2_{\Xs} := \norm{u}_{\Ys}^2 + \norm{\partial_t u}_{\Ys^\prime}^2\qqqq \forall u \in \Xs.
$$
The reader is referred to \cite[Chapter~10]{GR2011} for more details on function spaces defined on a space-time domain. The space $\Xs_0$ is a natural space for the weak solution to the problem \eqref{eq: IBVP}. 

Under mild assumptions on $\Gamma^\ast$ (see \cite[Chapter~10]{GR2011} and \cite{LR2013} for more details), there exist continuous space-time trace operators $\gamma_i: \Hs^{1,0}\left(Q_i\right)\rightarrow \Ls^2\left(\Gamma^\ast\right)\ (i=1,2)$, allowing us to define the space-time jump operator 
\begin{equation}
    \label{eq: space-time jump}
    \left[u\right]_\ast := \gamma_1 u - \gamma_2 u \qqqq \forall u\in \Hs^{1,0}\left(Q_1\cup Q_2\right).
\end{equation}
In what follows, we use $C > 0$ as a generic constant that is independent of the mesh parameter $h$ and the solution $u$ but may depend on the space-time domain $Q_T$, the norm $\norm{\vb}_{\LLs^\infty\left(Q_T\right)}$, and the coefficient $\kappa$. Their different values in different places are allowed.

\subsection{Well-posedness}

Given a source term $f\in \Ys^\prime$, the variational formulation of \eqref{eq: IBVP} reads as: Determine $u \in \Xs_0$ that satisfies
\begin{equation}
    \label{eq:vf}
    a\left(u, \vphi\right) = \inprod{f, \vphi}_{\Ys^\prime \times \Ys} \qqqq \forall \vphi \in \Ys,
\end{equation}
where the bilinear form $a : \Xs_0\times \Ys \to \R$ is defined by
$$
a\left(u, \vphi\right) := \inprod{\partial_t u, \vphi}_{\Ys^\prime \times \Ys} +\int\limits_0^T\int\limits_\Omega  \left(\vb\cdot\nabla u\right)\vphi + \kappa \nabla u \cdot \nabla \vphi  \dx \dt.
$$
In this section, we invoke the Banach-Ne{\v c}as-Babu\v{s}ka theorem to prove the well-posedness of the problem \eqref{eq:vf}. To that end, the following two results are necessary.

\begin{lemma}
    \label{lem: inf_sup}
    There exists a constant $C > 0$ such that
    \begin{equation}
    \label{eq:inf_sup}
        \sup_{\vphi \in \Ys \setminus \brac{0}} \dfrac{a\left(u, \vphi\right)}{\norm{\vphi}_{\Ys}} \ge C \norm{u}_{\Xs} \qqqq \fa u \in \Xs_0.
    \end{equation}
\end{lemma}

\begin{proof}
    For each $u \in \Xs_0$, we denote by $z \in \Ys$ the unique solution to the auxiliary equation
    \begin{equation}
        \label{eq:inf_sup 1}
        \int\limits_0^T\int\limits_\Omega \kappa \nabla z \cdot \nabla \phi \dx \dt = \inprod{\partial_t u, \phi}_{\Ys^\prime \times \Ys} \qqqq \fa \phi \in \Ys.
    \end{equation}
    From Riesz's representation theorem, it is clear that $\norm{z}_{\Ys} = \norm{\partial_t u}_{\Ys^\prime}$. The idea is to choose $\vphi = z + \lambda u \in \Ys$ in \eqref{eq:inf_sup} with sufficiently large $\lambda > 0$. We immediately have
    $$
    a\left(u, \vphi\right) = a\left(u, z + \lambda u\right) = a\left(u, z\right) + \lambda a\left(u, u\right).
    $$
    For the first term, we use the Cauchy inequality to arrive at
    $$
    \begin{aligned}
        a\left(u, z\right) &= \int\limits_0^T\int\limits_\Omega \kappa \nabla z \cdot \nabla z \dx \dt + \int\limits_0^T\int\limits_\Omega \left(\vb\cdot\nabla u\right)z + \kappa \nabla u \cdot \nabla z \dx \dt \\
        &\ge \norm{z}^2_{\Ys} - \norm{\vb}_{\LLs^\infty\left(Q_T\right)}\left(\dfrac{C}{4\veps}\norm{u}^2_{\Ys} + \veps\norm{z}^2_{\Ls^2\left(Q_T\right)}\right) - \left(\dfrac{1}{4\veps}\norm{u}^2_{\Ys} + \veps\norm{z}^2_{\Ys}\right)\\
        &\ge \left(1 - C\veps\right) \norm{z}^2_{\Ys} - \dfrac{C}{\veps} \norm{u}^2_{\Ys},
    \end{aligned}
    $$
    for any $\veps>0$. We invoked the inequality $\norm{z}_{\Ls^2\left(Q_T\right)}\le C \norm{z}_{\Ys}$ for all $z\in \Ys$ in the final step. In handling the term $a\left(u,u\right)$, we first deal with the advection part by utilizing the divergence theorem. Noting that $\nabla\cdot \vb (\xb, t)= 0$ for all $\left(\xb, t\right)\in \Omega\times [0,T]$ and $u = 0$ on $\pa\Om \times (0, T)$, we have
    \begin{equation}
        \label{eq: inf_sup 2}
        \begin{aligned}
            \int\limits_0^T\int\limits_\Omega \left(\vb\cdot \nabla u\right) u \dx\dt 
            & = \dfrac{1}{2}\int\limits_0^T\int\limits_\Omega \nabla\cdot \left(u^2\vb\right)-u^2 \left(\nabla\cdot \vb\right)\dx\dt \\
            & = \dfrac{1}{2}\int\limits_0^T\int\limits_\Omega \nabla\cdot \left(u^2\vb\right)\dx\dt \\
            & = \dfrac{1}{2}\int\limits_0^T \int\limits_{\partial\Omega} u^2\vb\cdot\nv_{\Omega}\ds\dt \\
            & = 0,
        \end{aligned}
    \end{equation}
    with $\nv_{\Omega}$ the unit outward normal to $\partial\Omega$. Then, we apply the integration by parts formula \cite[Lemma~7.3]{Roubicek2005} to get
    $$
    \begin{aligned}
        a\left(u, u\right) &= \inprod{\partial_t u, u}_{\Ys^\prime \times \Ys} +\int\limits_0^T\int\limits_\Omega \left(\vb\cdot \nabla u\right) u + \kappa \nabla u \cdot \nabla u \dx \dt\\
        &=\dfrac{1}{2}\norm{u\left(\cdot, T\right)}^2_{\Ls^2\left(\Omega\right)} - \dfrac{1}{2}\norm{u\left(\cdot, 0\right)}^2_{\Ls^2\left(\Omega\right)} + \norm{u}_{\Ys}^2\\
        &\ge \norm{u}_{\Ys}^2,
    \end{aligned}
    $$
    owing to $u\left(\cdot,0\right)=0$. Therefore, we end up with
    $$
    a\left(u, \vphi\right) \ge \left(1 - C\veps\right) \norm{z}^2_{\Ys} + \paren{\lambda - \dfrac{C}{\veps}} \norm{u}^2_{\Ys}=\left(1 - C\veps\right)\norm{\partial_t u}^2_{\Ys^\prime} + \paren{\lambda - \dfrac{C}{\veps}} \norm{u}^2_{\Ys}.
    $$
    We hence fix a small enough $\veps > 0$, then select an appropriate $\lambda > 0$ to obtain
    $$
    a\left(u, \vphi\right) \ge C \norm{u}^2_{\Xs}.
    $$
    In addition, we have 
    $$
    \norm{\vphi}_{\Ys} = \norm{z + \lambda u}_{\Ys} \le \norm{z}_{\Ys} + \lambda \norm{u}_{\Ys} = \norm{\partial_t u}_{\Ys^\prime} + \lambda \norm{u}_{\Ys} \le C\norm{u}_{\Xs},
    $$
    which leads to 
    $$
    \sup_{\vphi \in \Ys \setminus \brac{0}} \dfrac{a(u, \vphi)}{\norm{\vphi}_{\Ys}} \ge  C\norm{u}_{\Xs} \qqqq \fa u \in \Xs_0.
    $$
\end{proof} 

\begin{lemma}
    \label{lem: bnb2}
    If $a(u, \vphi)=0$ for all $u\in \Xs_0$, then $\vphi= 0$ in $\Ys$.
\end{lemma}

\begin{proof}
    For an arbitrary $u \in \Cs_0^1\left(Q_T\right)\subset \Xs_0$ and $\vphi \in \Ys$, the classical argument yields
    $$\inprod{\partial_t u, \vphi}_{\Ys^\prime \times \Ys} =\int\limits_0^T\int\limits_\Omega \partial_t u\vphi \dx\dt.$$
    In particular, by taking $\vphi \in \Ys$ such that
    \begin{equation}
        \label{eq: bnb2 1}
        a\left(u, \vphi\right)= \inprod{\partial_t u, \vphi}_{\Ys^\prime \times \Ys} +\int\limits_0^T\int\limits_\Omega \left(\vb\cdot\nabla u\right)\vphi + \kappa \nabla u \cdot \nabla \vphi \dx \dt = 0 \qqq \forall u \in \Xs_0,
    \end{equation}
    we arrive at
    $$\int\limits_0^T\int\limits_\Omega \partial_t u\vphi \dx\dt =  -\int\limits_0^T\int\limits_\Omega \left(\vb\cdot\nabla u\right)\vphi + \kappa \nabla u \cdot \nabla \vphi \dx \dt  \qqq \forall u \in \Cs_0^1\left(Q_T\right).$$
    For $\vphi\in \Ys$ and $u \in \Cs_0^1\left(Q_T\right)$, the left-hand side defines the distributional derivative $-\partial_t \vphi \left(u\right)$. Moreover, as the right-hand side is a bounded linear functional of $u$ on $\Ys$, we imply that $\partial_t \vphi\in \Ys^\prime$, i.e., $\vphi \in \Xs$. Then, we use the density of $\Cs_0^1(Q_T)$ in $\Ys$ to obtain
    \begin{equation}
        \label{eq: bnb2 2}
        -\inprod{\partial_t \vphi, u}_{\Ys^\prime \times \Ys} =  -\int\limits_0^T\int\limits_\Omega \left(\vb\cdot\nabla u\right)\vphi + \kappa \nabla u \cdot \nabla \vphi \dx \dt \qqq \forall u \in \Ys.
    \end{equation}
    Especially, for all $u\in \Xs_0$, we subtract \eqref{eq: bnb2 1} from \eqref{eq: bnb2 2} and invoke the integration by parts formula to get $\vphi\left(\cdot, T\right)=0$. Next, we choose $u = \vphi \in \Xs$ in \eqref{eq: bnb2 2} to have
    $$\dfrac{1}{2}\norm{\vphi\left(\cdot, 0\right)}^2_{\Ls^2(\Om)}+\int\limits_0^T\int\limits_\Omega \left(\vb\cdot\nabla \vphi\right)\vphi \dx\dt + \norm{\vphi}^2_{\Ys}=0,$$
    which means that $\vphi = 0\in \Ys$. Note that the second term vanishes by using  \eqref{eq: inf_sup 2}. We finish the proof.
\end{proof}

The well-posedness of the variational problem \eqref{eq:vf} is a direct consequence of Lemmas \ref{lem: inf_sup} and \ref{lem: bnb2}, according to the Banach-Ne{\v c}as-Babu\v{s}ka theorem \cite[Theorem~25.9]{Ern2021b}.

\begin{theorem}
    Given $f\in \Ys^\prime$. There exists a unique solution $u 
    \in \Xs_0$ to the problem \eqref{eq:vf} such that 
    $$
    \norm{u}_{\Xs}\le C \norm{f}_{\Ys^\prime}.
    $$
\end{theorem}

\section{Finite element discretization}
\label{sec: discretization}

\subsection{Space-time interface-fitted method}
\label{subsec: interface-fitted method}

Let us assume that $\Omega$ is a polyhedron in $\mathbb{R}^{d}$, with $d = 1$ or $2$. We denote by $\mathcal{T}_h$ a space-time interface-fitted decomposition of $Q_T = \Om \times (0, T)$ into shape-regular simplicial finite elements with the mesh size $h$ \cite{CZ1998}. Hence, every triangle or tetrahedron $K \in \mathcal{T}_h$ falls into one of the following scenarios (see {Figure~\ref{fig:K}}):

\begin{enumerate}
    \item $K \subset \overline{Q_1}$;
    \item $K \subset \overline{Q_2}$;
    \item $K \cap Q_1 \neq \varnothing$ and $K \cap Q_2 \neq \varnothing$. In this case, $d+1$ vertices of $K$ lie on $\Gamma^\ast$.
\end{enumerate}
\begin{figure}[!htp]
    \centering
    \subfigure[A space-time domain $Q_T$.]{\includegraphics[width=0.4\linewidth]{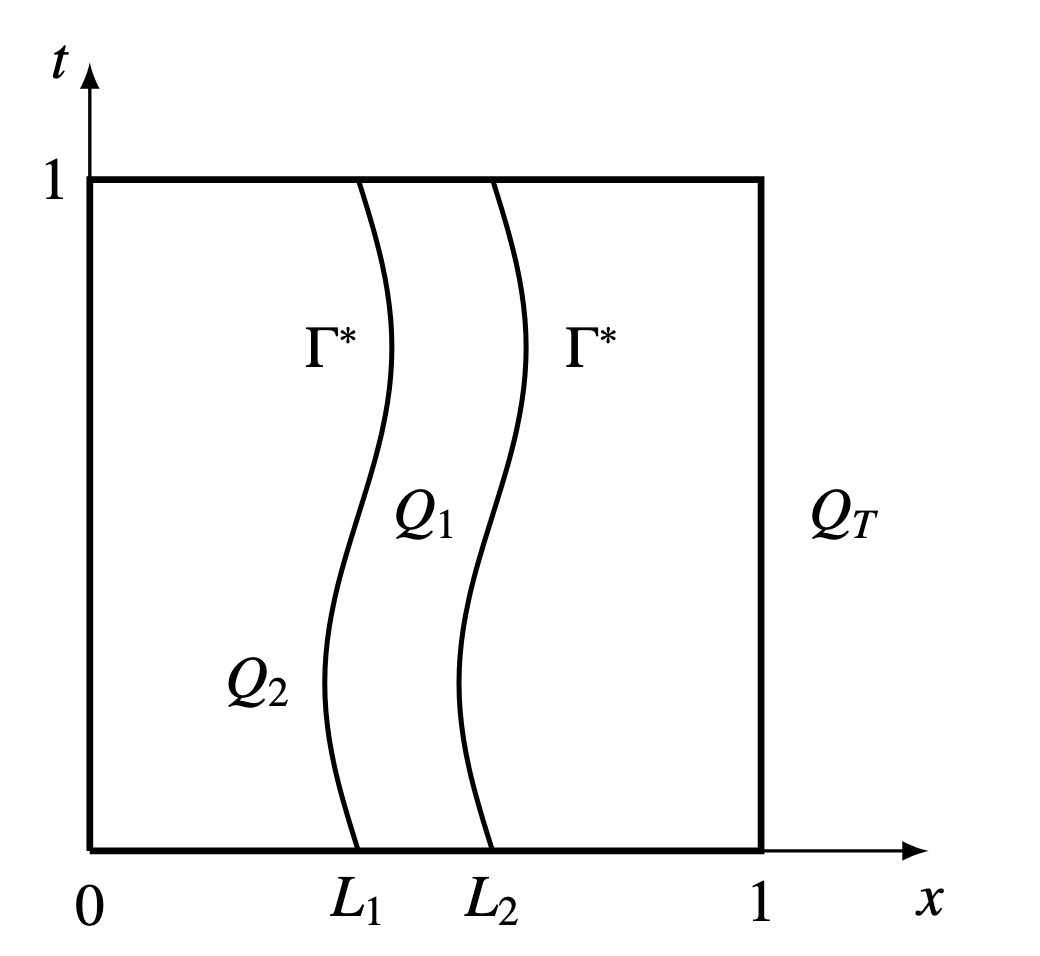}} \q
    \subfigure[The domain $Q_T$ after discretization.]{\includegraphics[width=0.5\linewidth]{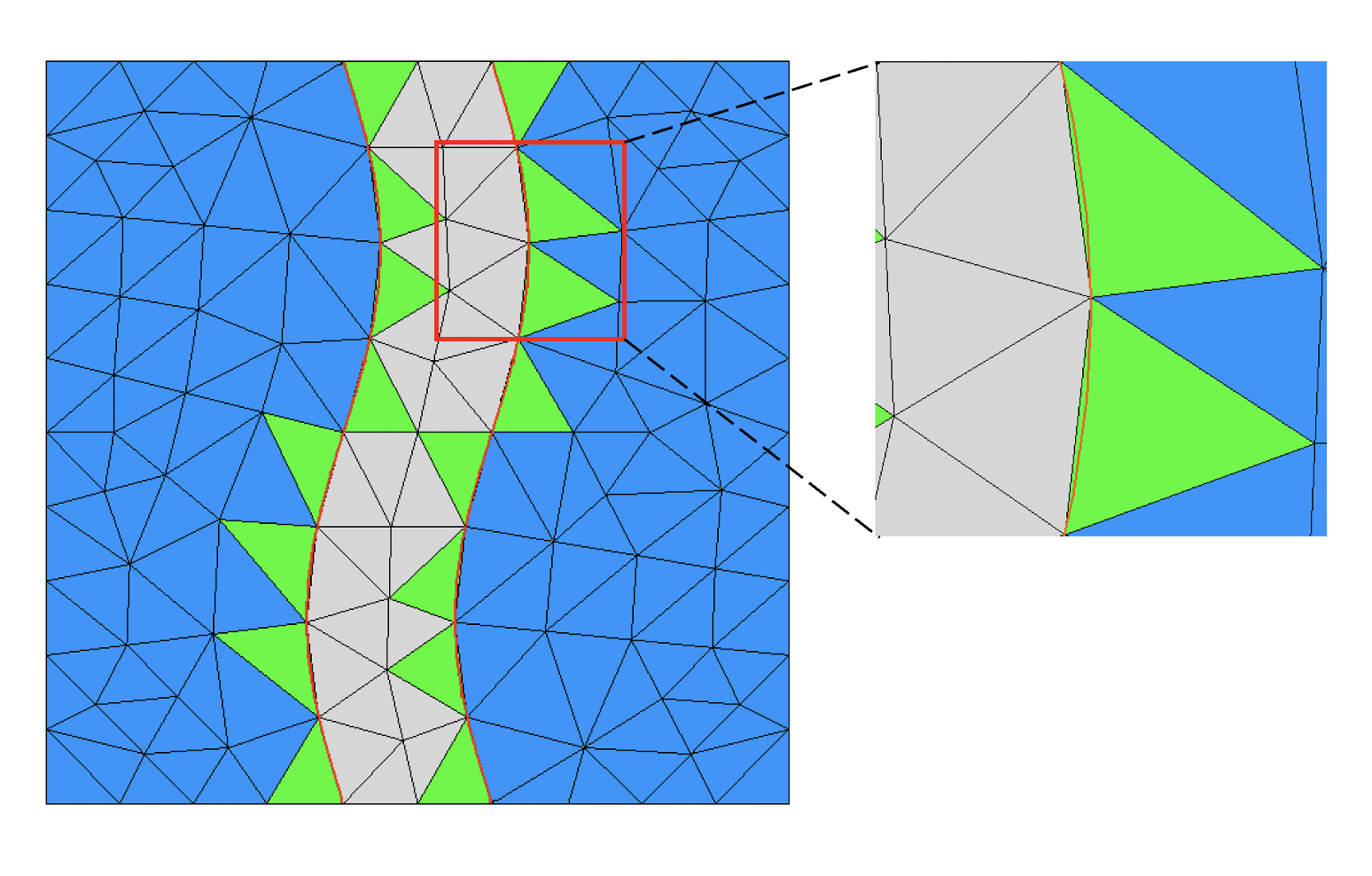}} 
    \caption{Illustration of a space-time domain $Q_T\subset \mathbb{R}^2$ and its space-time interface-fitted decomposition: The gray elements are in $\overline{Q_1}$, the blue elements are in $\overline{Q_2}$, and the green ones intersect both $Q_1$ and $Q_2$.}
    \label{fig:K}
\end{figure}

Moreover, we suppose that the mesh $\mathcal{T}_h$ is quasi-uniform \cite[Definition~22.20]{Ern2021c}. We denote by $\Gamma_h^\ast$ the linear approximation to $\Gamma^\ast$, consisting of edges or faces with all vertices lying on $\Gamma^\ast$. The space-time domain $Q_T$ is separated by $\Gamma_h^\ast$ into two regions $Q_{1,h}$ and $Q_{2,h}$, which linearly approximate $Q_1$ and $Q_2$, respectively.

Let $\Vs_h$ be the finite element space of continuous element-wise linear functions on $\mathcal{T}_h$ with zero values on $\partial \Omega\times (0,T)$ and $\Omega \times \{0\}$. Obviously, $\Vs_h\subset \Xs_0$ and $\Vs_h\subset \Ys$. We consider the following discrete variational problem of \eqref{eq:vf}: Identify $u_h\in \Vs_h$ that solves
\begin{equation}
    \label{eq: discrete vf}
    a_h\left(u_h, \vphi_h\right) = \inprod{f, \vphi_h}_{\Ys^\prime \times \Ys} \qqqq \forall \vphi_h\in \Vs_h,
\end{equation}
where the bilinear form $a_h: \Xs_0 \times \Ys \to \mathbb{R}$ is defined by
$$
a_h\left(u, \vphi\right) := \inprod{\partial_t u, \vphi}_{\Ys^\prime \times \Ys} + \int\limits_0^T\int\limits_\Omega \left(\vb\cdot \nabla u\right)\vphi +  \kappa_h \nabla u \cdot \nabla \vphi \dx \dt,
$$
with $\kappa_h$ the mesh-dependent approximation of $\kappa$
$$
\kappa_h\left(\xb, t\right) := 
\begin{cases}
    \kappa_{1} & \text {for } \, (\xb, t) \in Q_{1,h}, \\ 
    \kappa_{2} & \text {for } \, (\xb, t) \in Q_{2,h}.
\end{cases}
$$
The bilinear form $a_h\left(\cdot, \cdot\right)$ differs from $a\left(\cdot, \cdot\right)$ in the coefficient $\kappa_h$ substituting $\kappa$. For establishing the solvability of the problem \eqref{eq: discrete vf}, we introduce the following seminorm associated with $\kappa_h$
$$
\vertiii{v}^2 : = \sum_{i=1}^2\int\limits_{Q_{i,h}} \kappa_i \abs{\nabla v}^2 \dx \dt \qqqq \fa v \in \Hs^{1,0}\left(Q_{1,h}\cup Q_{2,h}\right).
$$
Please note that this seminorm is defined on the space $\Hs^{1,0}\left(Q_{1,h}\cup Q_{2,h}\right)$, which allows a jump of functions across the interface $\Gm^\ast$. When restricted on $\Ys$, it defines an equivalent norm, also denoted by $\vertiii{\cdot}$. In practice, this norm is more favorable than $\norm{\cdot}_{\Ys}$ since it involves the approximation $\kappa_h$, which is a constant on each element. Next, we consider the following auxiliary problem: For $v \in \Hs^{1}\left(Q_{1,h}\cup Q_{2,h}\right)$, determine $z_h\left(v\right)\in \Vs_h$ such that
\begin{equation}
    \label{eq: auxiliary elliptic}
    \int\limits_0^T\int\limits_\Omega \kappa_h \nabla z_h\left(v\right) \cdot \nabla \phi_h \dx \dt = \sum_{i=1}^2\int\limits_{Q_{i,h}} \partial_t v \, \phi_h \dx\dt \qqqq \forall \phi_h\in \Vs_h.
\end{equation}
It is clear that the problem \eqref{eq: auxiliary elliptic} has a unique solution. Choosing $\phi_h = z_h\left(v\right)$ gives us
$$
\vertiii{z_h \left(v\right)}^2 = \sum_{i=1}^2\int\limits_{Q_{i,h}} \partial_t v \, z_h \left(v\right) \dx\dt \le \norm{\partial_tv}_{\Ls^2\left(Q_{1,h}\cup Q_{2,h}\right)}\norm{z_h\left(v\right)}_{\Ls^2\left(Q_T\right)}\le C\norm{\partial_tv}_{\Ls^2\left(Q_{1,h}\cup Q_{2,h}\right)}\vertiii{z_h \left(v\right)},
$$
where we invoked the following property in the last step
\begin{equation}
    \label{eq: L2 and Y norm comparision}
    \norm{v}_{\Ls^2\left(Q_T\right)}\le C \vertiii{v} \qqqq \forall v \in \Ys.
\end{equation}
We arrive at the useful inequality
\begin{equation}
    \label{eq: discrete and continuous norm comparision}
    \vertiii{z_h\left(v\right)} \le C\norm{\partial_tv}_{\Ls^2\left(Q_{1,h}\cup Q_{2,h}\right)}\qqqq \forall v \in \Hs^{1}\left(Q_{1,h}\cup Q_{2,h}\right).
\end{equation}
Now, we can introduce the following mesh-dependent norm analog to $\norm{\cdot}_{\Xs}$
\begin{equation}
    \label{eq: discrete norms 2}
    \vertiii{v}^2_\ast := \vertiii{v}^2 + \vertiii{z_h\left(v\right)}^2 \qqqq \fa v \in \Hs^{1}\left(Q_{1,h}\cup Q_{2,h}\right).
\end{equation}
Hence, the stability condition for the bilinear form $a_h\left(\cdot, \cdot\right)$ shall be established with respect to the discrete norm $\vertiii{\cdot}_\ast$. The following result is a discrete version of the continuous inf-sup condition \eqref{eq:inf_sup}.

\begin{lemma}
There exists a constant $C > 0$ such that
    \begin{equation}
    \label{eq: discrete inf-sup}
        \sup_{\vphi_h \in \Vs_h\setminus \{0\}} \dfrac{a_h\left(u_h, \vphi_h\right)}{\vertiii{\vphi_h}} \ge C \vertiii{u_h}_\ast \qqqq \fa u_h \in \Vs_h.
    \end{equation}
\end{lemma}
                    
\begin{proof}
    The proof is analogous to the continuous problem. For all $u_h,\vphi_h\in\Vs_h$, it holds that
    $$
    a_h\left(u_h, \vphi_h\right) = \int\limits_0^T\int\limits_\Omega \left(\partial_t u_h + \vb\cdot \nabla u_h\right)\vphi_h +  \kappa_h \nabla u_h \cdot \nabla \vphi_h \dx \dt.
    $$
    We choose $\vphi_h := z_h + \lambda u_h$ for some sufficiently large $\lambda > 0$ and arrive at
    $$
    a_h\left(u_h, \vphi_h\right) = a_h\left(u_h, z_h\right) + \lambda a_h\left(u_h, u_h\right),
    $$
    where $z_h := z_h\left(u_h\right)\in \Vs_h$ is the solution to the problem \eqref{eq: auxiliary elliptic}. By using the Cauchy inequality and \eqref{eq: L2 and Y norm comparision}, we have
    $$
    \begin{aligned}
        a_h\left(u_h, z_h\right) &= \vertiii{z_h}^2 + \int\limits_0^T\int\limits_\Omega \left(\vb\cdot\nabla u_h\right)z_h + \kappa_h \nabla u_h \cdot \nabla z_h \dx \dt \\
        &\ge \vertiii{z_h}^2 - \norm{\vb}_{\LLs^\infty\left(Q_T\right)}\left(\dfrac{C}{4\veps}\vertiii{u_h}^2 + \veps\norm{z_h}^2_{\Ls^2\left(Q_T\right)}\right) - \left(\dfrac{1}{4\veps}\vertiii{u_h}^2 + \veps\vertiii{z_h}^2\right)\\
        &\ge \left(1 - C\veps\right) \vertiii{z_h}^2 - \dfrac{C}{\veps} \vertiii{u_h}^2,
    \end{aligned}
    $$
    for any $\veps > 0$. On the other hand, we apply the integration by parts formula and proceed as in \eqref{eq: inf_sup 2} to obtain
    $$
    a_h\left(u_h, u_h\right) =\dfrac{1}{2}\norm{u_h\left(\cdot, T\right)}^2_{\Ls^2\left(\Omega\right)} + \vertiii{u_h}^2\ge \vertiii{u_h}^2.
    $$
    Therefore, we can bound the discrete bilinear form $a_h\left(u_h, \vphi_h\right)$ as follows
    $$
    a_h\left(u_h, \vphi_h\right) \ge \left(1 - C\veps\right) \vertiii{z_h}^2 + \paren{\lambda - \dfrac{C}{\veps}} \vertiii{u_h}^2 \ge C\vertiii{u_h}^2_\ast.
    $$
    Finally, we invoke the estimate $\vertiii{\vphi_h} = \vertiii{z_h+\lambda u_h}\le \vertiii{z_h} + \lambda \vertiii{u_h}\le C \vertiii{u_h}_\ast$ to complete the proof.
\end{proof}

The unique solvability of the discrete problem \eqref{eq: discrete vf} follows from the discrete Banach-Ne{\v c}as-Babu\v{s}ka theorem \cite[Theorem~26.6]{Ern2021b}.

\begin{theorem}
    Given $f\in \Ys^{\prime}$. There exists a unique solution $u_h\in \Vs_h$ to the problem \eqref{eq: discrete vf}.
\end{theorem}

\subsection{Auxiliary results}

We emphasize again that the bilinear form $a_h\left(\cdot, \cdot\right)$ is non-conformal to $a\left(\cdot, \cdot\right)$, since the mesh $\mathcal{T}_h$ is non-conformal to $Q_1$ and $Q_2$ in the discrepancy region. In order to establish the desired a priori error estimates, we need some auxiliary estimates of functions in this region. For $i =1,2$, the mismatch between $Q_i$ and $Q_{i,h}$ is defined by
$$
S_h^1 := Q_{1,h}\setminus \ovl{Q_1} = Q_2\setminus \ovl{Q_{2,h}},\qqqq S_h^2 := Q_{2,h}\setminus \ovl{Q_2} = Q_1\setminus \ovl{Q_{1,h}},
$$
and $S_h := S_h^1\cup S_h^2$ (see Figure~\ref{fig: mismatch region}). We define by $\mathcal{T}_h^\ast = \left\{K\in \mathcal{T}_h\mid K\cap \Gamma^\ast \ne \varnothing\right\}$ the set of all interface elements. The following lemma collects some geometry approximation properties of interface elements from Lemma 3 in \cite{BK1996} and Lemma 3.3.4 in \cite{Feis1987}. The results were proved with $d=1$ and can be extended to the case $d=2$.

\begin{figure}[!htp]
    \centering
    \includegraphics[width=0.3\textwidth]{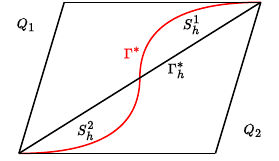}
    \caption{Illustration of the discrepancy region $S_h = S_h^1\cup S_h^2$ with $d=1$. The red curve represents the space-time interface $\Gamma^\ast$, the black diagonal represents the discrete interface $\Gamma_h^\ast$.}
    \label{fig: mismatch region}
\end{figure}

\begin{lemma}
    \label{lem: geometry appximation}
    Let $K\in \mathcal{T}^{\ast}_h$ and assume that $\Gamma^\ast$ is of class $\Cs^2$. Then, the mesh $\mathcal{T}_h$ resolves the space-time interface $\Gamma^\ast$ to the second order, i.e.,
    $$
    \operatorname{dist}\left(\Gamma_h^\ast\cap K; \Gamma^\ast\cap K\right) \le Ch^2.
    $$
    For all $u\in \Hs^1\left(S_h\right)$, there exists a constant $C > 0$ independent of $u$ and $h$ such that
    \begin{equation}
        \label{eq: poincare inequality}
        \norm{u}_{\Ls^2\left(S_h\cap K\right)} \le C\left(h\norm{u}_{\Ls^2\left(\Gamma^\ast\cap K\right)} + h^2\norm{\Ds u}_{\LLs^2\left(S_h\cap K\right)}\right),
    \end{equation}
    where we see the trace of $u$ on $\Gamma^\ast$ from $S_h$ and $\Ds :=\left(\nabla , \partial_t \right)^\top$ is the space-time gradient operator. It holds for the area or volume of the discrepancy region $S_h$ that
    \begin{equation}
        \label{eq: discrepancy region}
        \abs{S_h} \le C h^{2}.
    \end{equation}
\end{lemma}

A technical difficulty arises from the mismatch between $\mathcal{T}_h$ and the subdomains $Q_1$ and $Q_2$ is that the gradient of $u$ jumps across $\Gamma^\ast$ whereas the discrete functions exhibit a similar discontinuity over the approximated interface $\Gamma_h^\ast$. To deal with the mismatch and analyze the bilinear forms in the discrepancy region $S_h$, we follow the ideas in \cite{FJR2023} employing some appropriate extension operators. In particular, for any fixed $s \ge 0$, let $u \in \Hs^s\left(Q_1 \cup Q_2\right)$ and $u_i:=u_{\mid Q_i} \in \Hs^s\left(Q_i\right)$ be its restriction to the subdomain $Q_i\ (i=1,2)$. We employ the smooth Stein extensions $\tilde{u}_i \in \Hs^s\left(Q_T\right)\ (i=1,2)$ to the whole domain $Q_T$. These extensions have the following properties
\begin{equation}
    \label{eq: extension operator}
    \tilde{u}_i=u_i \quad \text { in } Q_i, \quad \norm{\tilde{u}_i}_{\Hs^s\left(Q_T\right)} \leq C\norm{u_i}_{\Hs^s\left(Q_i\right)}, \quad i=1,2,
\end{equation}
given that $\Gamma^\ast$ is Lipschitz continuous (see \cite[Section~VI.3.1]{Stein1971}). We define a map $\pi: \Hs^s\left(Q_1 \cup Q_2\right) \rightarrow \Hs^s\left(Q_{1,h} \cup Q_{2,h}\right)$ by
\begin{equation}
    \label{eq: auxiliary function}
    \pi u \left(\xb, t\right) := 
    \begin{cases}
        \tilde{u}_1\left(\xb, t\right)  & \text{for } \, \left(\xb, t\right) \in Q_{1,h}, \\
        \tilde{u}_2\left(\xb, t\right)  & \text{for } \,\left(\xb, t\right) \in Q_{2,h} .
    \end{cases}
\end{equation}
Please note that $\pi u$ can be discontinuous across $\Gamma_h^\ast$. The following lemma looks at the difference between $u$ and $\pi u$. By definition, $u - \pi u$ vanishes everywhere except $S_h$, and it belongs to $\Hs^s(S_h)$. 

\begin{lemma}
    Let $u\in \Hs^2\left(Q_1\cup Q_2\right)$ and $\pi u \in \Hs^2\left(Q_{1,h}\cup Q_{2,h}\right)$ be the function defined in \eqref{eq: auxiliary function}. There exists a constant $C > 0$ such that
    \begin{equation}
        \label{eq: discrepancy difference in u}
        \norm{\Ds\left(u-\pi u\right)}_{\LLs^2\left(S_h\right)}\le C h \norm{u}_{\Hs^2\left(Q_1\cup Q_2\right)}.
    \end{equation}
\end{lemma}

\begin{proof}
    For all $K\in \mathcal{T}_h^\ast$, we use \eqref{eq: poincare inequality} to obtain
    $$
    \norm{\Ds\left(u-\pi u\right)}^2_{\LLs^2\left(S_h\cap K\right)} \le C\left(h^2 \norm{\Ds\left(u-\pi u\right)}^2_{\LLs^2\left(\Gamma^\ast\cap K\right)} + h^4 \norm{\Ds^2\left(u-\pi u\right)}^2_{\LLs^2\left(S_h\cap K\right)}\right).
    $$
    We sum over all elements $K\in \mathcal{T}_h^\ast$. For the integral on $\Gamma^\ast$, we apply the trace inequality and \eqref{eq: extension operator} to arrive at
    $$
    \norm{\Ds\left(u-\pi u\right)}^2_{\LLs^2\left(\Gamma^\ast\right)} \le 2\left(\norm{\Ds u}^2_{\LLs^2\left(\Gamma^\ast\right)} + \norm{\Ds\pi u}^2_{\LLs^2\left(\Gamma^\ast\right)}\right)\le C \norm{u}^2_{\Hs^2\left(Q_1\cup Q_2\right)},
    $$
    where the constant $C$ depends on $Q_1$ and $Q_2$, but not on $Q_{1, h}$ and $Q_{2, h}$. We finish the proof.
\end{proof}

We also recall the interpolation estimate for a Lagrangian interpolant. Note that for $u\in \Hs^2\left(Q_1 \cup Q_2\right)$, the Sobolev embedding \cite[Theorem~2.35]{Ern2021c} implies that $u\in \Cs\left(\overline{Q_1}\right)\cap \Cs\left(\overline{Q_2}\right)$. In addition, if the interface condition $[u]_\ast = 0$ is satisfied for almost all $\left(\xb, t\right) \in \Gamma^\ast$, we have $u \in \Cs\left(\overline{Q_T}\right)$. Here, $[u]_\ast$ denotes the space-time jump \eqref{eq: space-time jump} of $u$ across $\Gamma^\ast$. Let $I_h: \Cs\left(\overline{Q_T}\right) \rightarrow \Vs_h$ be the classical linear interpolation operator associated with the space $\Vs_h$. Since $\mathcal{T}_h$ matches perfectly to $\Gamma^\ast_h$, we can extend the classical theories in \cite[Section~3.1]{Ciarlet2002} to handle $\pi u$ without destroying the optimal order. 

\begin{lemma}
    \label{lem: interpolation error}
    For all $u\in \Hs^2\left(Q_1 \cup Q_2\right)$ such that $\left[u\right]_{\ast}=0$ and $\pi u \in \Hs^2\left(Q_{1,h}\cup Q_{2,h}\right)$, the interpolation operator $I_h$ satisfies the following approximability property
    \begin{equation}
        \label{eq: interpolation}
        \norm{\pi u - I_h u}_{\Ls^2\left(Q_T\right)} + h \norm{\Ds\left(\pi u - I_h u\right)}_{\LLs^2\left(Q_{1,h}\cup Q_{2,h}\right)} \le C h^2 \norm{u}_{\Hs^2\left(Q_1\cup Q_2\right)}.
    \end{equation}
\end{lemma}

\begin{proof}
    The proof of \eqref{eq: interpolation} is standard since in each subdomain $Q_{i,h}$ we have $I_h u = I_h \tilde{u}_i$ in all Lagrange nodes of the mesh $\mathcal{T}_h$. We apply the classical interpolation theory for $\tilde{u}_i\in \Hs^2\left(Q_T\right)$ and \eqref{eq: extension operator} to obtain the result.
\end{proof}

\begin{remark}
    In Lemma \ref{lem: interpolation error}, we consider $u\in \Hs^2\left(Q_1 \cup Q_2\right)$ such that $\left[u\right]_{\ast}=0$ to be able to use the interpolant $I_h: \Cs\left(\overline{Q_T}\right) \rightarrow \Vs_h$. The condition $[u]_\ast=0$ was also used in Lemma 2.4 of \cite{Deka2010} and Lemma 2.3 of \cite{DA2013}, where the authors arrived at an optimal approximation of $I_h$.
\end{remark}

\subsection{A priori error estimates}

In Section \ref{sec:var_form}, we proved the well-posedness of the problem \eqref{eq:vf} in the space $\Xs_0$, which only requires $\partial_t u\in \Ys^\prime$. However, one needs additional regularities of the solution $u$ to study the discretization error. More particularly, we assume that the solution $u\in \Xs_0$ of the problem \eqref{eq:vf} satisfies $u\in \Hs^{0,1}\left(Q_T\right)\cap\Hs^s\left(Q_1\cup Q_2\right)$, with $s > \frac{d+3}{2}$. Extending the ideas in \cite{FJR2023}, we study the error $u-u_h$ by choosing the discrete norm $\vertiii{\pi u-u_h}_\ast$. This choice is appropriate for employing numerical quadratures and allows us to better understand the mismatch's effect in $S_h$. 

\begin{theorem}
    \label{theo: error estimate H1.1}
    Let $u\in \Xs_0$ and $u_h\in \Vs_h$ be the solutions of the variational problems \eqref{eq:vf} and \eqref{eq: discrete vf}, respectively. If $u\in \Hs^{0,1}\left(Q_T\right)\cap\Hs^s\left(Q_1\cup Q_2\right)$ for given $s > \frac{d+3}{2}$, then there holds the following estimate
    \begin{equation}
        \label{eq: error estimate H1.1}
        \vertiii{\pi u-u_h}_\ast \le C h \norm{u}_{\Hs^s\left(Q_1\cup Q_2\right)}.
    \end{equation}
\end{theorem}

\begin{proof}
    First of all, we note that $u\in \Xs_0\cap \Hs^{0,1}\left(Q_T\right) \sst \Hs^{1}\left(Q_T\right)$, which implies that $[u]_\ast = 0$. In addition, if $u\in \Hs^s\left(Q_1\cup Q_2\right)$ with $s > \frac{d+3}{2}$, then we are able to employ the interpolant $I_h$. We start with the triangle inequality
    \begin{equation}
        \label{eq: error estimate H1 1}
        \vertiii{\pi u-u_h}_\ast \le \vertiii{\pi u- I_h u}_\ast + \vertiii{I_h u - u_h}_\ast.
    \end{equation}
    In what follows, we denote $w := \pi u- I_h u$ for convenience. By invoking \eqref{eq: discrete norms 2}, we have
    $$
    \vertiii{w}^2_\ast = \vertiii{w}^2 + \vertiii{z_h\left(w\right)}^2\le \max\left\{\kappa_{1}, \kappa_{2}\right\}\norm{\nabla w}^2_{\LLs^2\left(Q_{1,h}\cup Q_{2,h}\right)} + \vertiii{z_h\left(w\right)}^2.
    $$
    The first term is already bounded in \eqref{eq: interpolation}. For the second term, we employ \eqref{eq: discrete and continuous norm comparision} and \eqref{eq: interpolation} to obtain
    $$
    \vertiii{z_h\left(w\right)} \le C \norm{\Ds w}_{\LLs^{2}\left(Q_{1,h}\cup Q_{2,h}\right)}\le Ch \norm{u}_{\Hs^2\left(Q_1\cup Q_2\right)}.
    $$
    Thus
    \begin{equation}
        \label{eq: error estimate H1 3}
        \vertiii{w}_\ast \le Ch \norm{u}_{\Hs^2\left(Q_1\cup Q_2\right)}.
    \end{equation}
    Next, we estimate the second term on the right-hand side of \eqref{eq: error estimate H1 1} by using \eqref{eq: discrete inf-sup}. It follows that
    $$
    C\vertiii{I_h u - u_h}_\ast \le \sup_{\vphi_h \in \Vs_h\setminus \{0\}} \dfrac{a_h\left(I_h u - u_h, \vphi_h\right)}{\vertiii{\vphi_h}}= \sup_{\vphi_h \in \Vs_h\setminus \{0\}} \dfrac{a_h\left(I_h u, \vphi_h\right) - a\left(u, \vphi_h\right)}{\vertiii{\vphi_h}}=\sup_{\vphi_h \in \Vs_h\setminus \{0\}} \dfrac{I_1 - I_2}{\vertiii{\vphi_h}},
    $$
    where
    $$
        \begin{aligned}
            I_1&:=\sum_{i=1}^2\int\limits_{Q_{i,h}} \partial_t \, \tilde{u}_i \, \vphi_h + \vb\cdot \nabla \tilde{u}_i \, \vphi_h +  \kappa_i \nabla \tilde{u}_i \cdot \nabla \vphi_h \dx \dt - a\left(u, \vphi_h\right),\\
            I_2&:=\sum_{i=1}^2\int\limits_{Q_{i,h}} \partial_t \, \tilde{u}_i \, \vphi_h + \vb\cdot \nabla \tilde{u}_i \, \vphi_h +  \kappa_i \nabla \tilde{u}_i \cdot \nabla \vphi_h \dx \dt - a_h\left(I_h u, \vphi_h\right).
        \end{aligned}
    $$
    We first deal with $I_1$. Since $u\in \Hs^{0,1}\left(Q_T\right)$ and the associated integrand vanishes everywhere besides on $S_h$, we have 
    \begin{equation}
        \label{eq: error estimate H1 4}
        \begin{aligned}
            I_1 &= \int\limits_{S_h} \partial_t\left(\pi u - u\right) \vphi_h +\vb\cdot\nabla\left(\pi u - u\right) \vphi_h + \left(\kappa_h \nabla \left(\pi u\right) - \kappa \nabla u\right) \cdot \nabla \vphi_h \dx \dt\\
            &\le C {\norm{\Ds\left(\pi u - u\right)}_{\LLs^2\left(S_h\right)}}\norm{\vphi_h}_{\Ls^2\left(S_h\right)} + C\left(\norm{\nabla \pi u}_{\LLs^2\left(S_h\right)} + \norm{\nabla u}_{\LLs^2\left(S_h\right)}\right)\vertiii{\vphi_h}.
        \end{aligned}
    \end{equation}
    To treat the first part on the right-hand side, note that the term $\norm{\Ds\left(\pi u - u\right)}_{\LLs^2\left(S_h\right)}$ is bounded in \eqref{eq: discrepancy difference in u}, and from \eqref{eq: L2 and Y norm comparision}, we have
    $$
    \norm{\vphi_h}_{\Ls^2\left(S_h\right)}\le C\vertiii{\vphi_h} \qqqq \forall \vphi_h \in \Vs_h.
    $$
    Now, we handle the two norms $\norm{\nabla \pi u}_{\LLs^2\left(S_h\right)}$ and $\norm{\nabla u}_{\LLs^2\left(S_h\right)}$ by invoking the technique in Lemma 2.1 of \cite{CZ1998}. Take $\norm{\nabla \pi u}_{\LLs^2\left(S_h\right)}$ for instance. Please note that given $s > \frac{d+3}{2}$, the Sobolev embedding theorem says that $\HHs^{s-1}\left(Q_T\right) \hookrightarrow \hookrightarrow \CCs\left(\ovl{Q_T}\right)$. Therefore, we can use \eqref{eq: discrepancy region} and \eqref{eq: extension operator} to get
    $$
    \begin{aligned}
        \norm{\nabla \pi u}_{\LLs^2\left(S_h\right)} \le \abs{S_h}^{\frac{1}{2}} \norm{\nabla \pi u}_{\CCs\left(\ovl{S_h}\right)} &\le C h \max{\left\{\norm{\nabla \tilde{u}_1}_{\CCs\left(\ovl{Q_T}\right)}, \norm{\nabla \tilde{u}_2}_{\CCs\left(\ovl{Q_T}\right)}\right\}} \\
        & \le C h \max{\left\{\norm{\nabla \tilde{u}_1}_{\HHs^{s-1}\left(Q_T\right)}, \norm{\nabla \tilde{u}_2}_{\HHs^{s-1}\left(Q_T\right)}\right\}} \le C h \norm{u}_{\Hs^s\left(Q_1\cup Q_2\right)},
    \end{aligned}
    $$
    In this way, we can bound $I_1$ by
    $$
    I_1 \le Ch\norm{u}_{\Hs^s\left(Q_1\cup Q_2\right)}\vertiii{\vphi_h}.
    $$
    For estimating $I_2$, we utilize \eqref{eq: L2 and Y norm comparision} to arrive at
    $$
    \begin{aligned}
        I_2 &=\sum_{i=1}^2\int\limits_{Q_{i,h}} \partial_t \tilde{u}_i \, \vphi_h + \vb\cdot \nabla \tilde{u}_i \, \vphi_h +  \kappa_i \nabla \tilde{u}_i \cdot \nabla \vphi_h \dx \dt - a_h\left(I_h u, \vphi_h\right)\\
        &=\sum_{i=1}^2\int\limits_{Q_{i,h}} \partial_t w \, \vphi_h + \vb\cdot \nabla w \, \vphi_h + \kappa_i \nabla w \cdot \nabla \vphi_h \dx \dt\\
        &\le C \left({\norm{\Ds w}_{\LLs^2\left(Q_{1,h}\cup Q_{2,h}\right)}}\norm{\vphi_h}_{\Ls^2\left(Q_T\right)} + \norm{\nabla w}_{\LLs^2\left(Q_{1,h}\cup Q_{2,h}\right)}\vertiii{\vphi_h} \right)\\
        &\le C \norm{\Ds w}_{\LLs^2\left(Q_{1,h}\cup Q_{2,h}\right)} \vertiii{\vphi_h},
    \end{aligned}
    $$
    which, by means of \eqref{eq: interpolation}, implies that
    $$
    I_2 \le Ch \norm{u}_{\Hs^2\left(Q_1\cup Q_2\right)}\vertiii{\vphi_h}.
    $$
    Finally, we end up with
    \begin{equation}
        \label{eq: error estimate H1 8}
        \vertiii{I_h u - u_h}_\ast \le C h\norm{u}_{\Hs^s\left(Q_1\cup Q_2\right)}\vertiii{\vphi_h}.
    \end{equation}
    We combine \eqref{eq: error estimate H1 8} with \eqref{eq: error estimate H1 3} and \eqref{eq: error estimate H1 1} to obtain the result. The proof is complete.
\end{proof}

\begin{remark}
    From a computational perspective, higher-order finite elements can be freely used instead of being limited to space-time $P_1$ elements. However, the error achieved will not be better than the one using $P_1$ elements. The reason is that a priori error estimates depend not only on the type of finite elements used but also on the discretization error of the interface. Since $\abs{S_h} \le C h^{2}$ (cf. Lemma~\ref{lem: geometry appximation}), the orders of a priori error estimates will not improve even when higher-order finite elements are employed.
    
    To achieve better results, we need to invoke a higher-order interface-fitted approximation method, generalizing our strategy. This topic has been studied by Li et al. for elliptic interface problems \cite{LMWZ2010}. For our problem, we believe their method is well-suited to be combined with our method. However, to the best of our knowledge, extending their result to the problem \eqref{eq: IBVP} remains an open question and will be a potential direction for future research.
\end{remark}

In case of a weaker regularity condition $u \in \Hs^{0,1}\left(Q_T\right)\cap\Hs^2\left(Q_1\cup Q_2\right)$, we have the following a priori error estimate, which is now different between one and two spatial dimensions. 

\begin{theorem}
    \label{theo: error estimate H1}
    Let $u\in \Xs_0$ and $u_h\in \Vs_h$ be the solutions of the variational problems \eqref{eq:vf} and \eqref{eq: discrete vf}, respectively. Assume that $u\in \Hs^{0,1}\left(Q_T\right)\cap\Hs^2\left(Q_1\cup Q_2\right)$. Then, there exists $h_0 > 0$ such that for all $h < h_0$, the following estimate holds
    \begin{equation}
        \label{eq: error estimate H1}
        \vertiii{\pi u-u_h}_\ast \le C 
        \begin{Bmatrix}
            h\sqrt{\abs{\log h}}\\
            h^{\frac{2}{3}}
        \end{Bmatrix}\norm{u}_{\Hs^2\left(Q_1\cup Q_2\right)}\qqq \text{for }
        \begin{Bmatrix}
            d=1\\
            d=2
        \end{Bmatrix}.
    \end{equation}
\end{theorem}

\begin{proof}
    The proof is analogous to the previous theorem. The only differences are the estimates of $\norm{\nabla \pi u}_{\LLs^2\left(S_h\right)}$ and $\norm{\nabla u}_{\LLs^2\left(S_h\right)}$ in \eqref{eq: error estimate H1 4}. We consider the former term since the latter can be treated similarly. We recall that the embedding $\HHs^1\left(Q_T\right)\hookrightarrow \LLs^p\left(Q_T\right)$ holds true for all $p\in [1,\infty)$ when $d=1$ and for $p\in [1,6]$ when $d=2$. In the following, we consider $p \ge 2$ if $d=1$ and set $p = 6$ in the other case while using the same notation. One can get from \eqref{eq: discrepancy region} that
    $$
    \norm{\nabla \pi u}_{\LLs^2\left(S_h\right)} \le \abs{S_h}^{\frac{1}{2}-\frac{1}{p}} \norm{\nabla \pi u}_{\LLs^p\left(S_h\right)} \le C h^{1 - \frac{2}{p}} \norm{\nabla \pi u}_{\LLs^p\left(S_h\right)},
    $$
    which implies 
    \begin{equation}
        \label{eq: error estimate H1 6}
        \begin{aligned}
            \norm{\nabla \pi u}_{\LLs^2\left(S_h\right)}&\le C h^{1-\frac{2}{p}}\left(\norm{\nabla\tilde{u}_1}_{\LLs^p\left(Q_T\right)}^p + \norm{\nabla\tilde{u}_2}_{\LLs^p\left(Q_T\right)}^p\right)^{\frac{1}{p}}\\
            &\le C
            \begin{Bmatrix}
                h^{1-\frac{2}{p}}p^{\frac{1}{2}}\\
                h^{1-\frac{2}{6}}
            \end{Bmatrix}\left(\norm{\nabla\tilde{u}_1}_{\HHs^1\left(Q_T\right)}^p + \norm{\nabla\tilde{u}_2}_{\HHs^1\left(Q_T\right)}^p\right)^{\frac{1}{p}}\\
            &\le C
            \begin{Bmatrix}
                h^{1-\frac{2}{p}}p^{\frac{1}{2}}\\
                h^{1-\frac{2}{6}}
            \end{Bmatrix}\norm{u}_{\Hs^2\left(Q_1\cup Q_2\right)}\le C
            \begin{Bmatrix}
                h\sqrt{\abs{\log h}}\\
                h^{\frac{2}{3}}
            \end{Bmatrix}\norm{u}_{\Hs^2\left(Q_1\cup Q_2\right)}\qqq \text{for }
            \begin{Bmatrix}
                d=1\\
                d=2
            \end{Bmatrix}.
        \end{aligned}
    \end{equation}
    In this step, we employed from Lemma 2.1 in \cite{RW1994} the Sobolev inequality
    \begin{equation}
    \label{eq:Sobolev_inequality}
        \norm{\nabla \tilde{u}_i}_{\LLs^p\left(Q_T\right)}\le C p^{\frac{1}{2}}\norm{\nabla \tilde{u}_i}_{\HHs^1\left(Q_T\right)},
    \end{equation}
    which holds true for all $p \in \left[2, \infty\right)$ when $d=1$, with $C$ independent of $p$. Moreover, one can easily prove that there exists $h_0>0$ such that for all $h < h_0$, the function $f\left(p\right)=h^{-\frac{2}{p}}p^{\frac{1}{2}}$ achieves its global minimum at $p=-4\log h$. Therefore, we arrive at
    $$
    \begin{aligned}
        I_1 &= \int\limits_{S_h} \partial_t\left(\pi u - u\right) \vphi_h +\vb\cdot\nabla\left(\pi u - u\right) \vphi_h + \left(\kappa_h \nabla \left(\pi u\right) - \kappa \nabla u\right) \cdot \nabla \vphi_h \dx \dt\\
        & \le C
        \begin{Bmatrix}
            h\sqrt{\abs{\log h}}\\
            h^{\frac{2}{3}}
        \end{Bmatrix}\norm{u}_{\Hs^2\left(Q_1\cup Q_2\right)}\vertiii{\vphi_h}\qqq \text{for }
        \begin{Bmatrix}
            d=1\\
            d=2
        \end{Bmatrix}.
    \end{aligned}
    $$
    Hence, we have
    \begin{equation}
        \label{eq: error estimate H1 7}
        \vertiii{I_h u - u_h}_\ast \le C
        \begin{Bmatrix}
            h\sqrt{\abs{\log h}}\\
            h^{\frac{2}{3}}
        \end{Bmatrix}\norm{u}_{\Hs^2\left(Q_1\cup Q_2\right)}\qqq \text{for }
        \begin{Bmatrix}
            d=1\\
            d=2
        \end{Bmatrix}.
    \end{equation}
    The estimate \eqref{eq: error estimate H1} follows by substituting \eqref{eq: error estimate H1 3} and \eqref{eq: error estimate H1 7} into \eqref{eq: error estimate H1 1}.
\end{proof}

\begin{remark}
    Providing that $u\in \Hs^{0,1}\left(Q_T\right)\cap\Hs^s\left(Q_1\cup Q_2\right)$ with $s>\frac{d+3}{2}$, we obtained an optimal error estimate with respect to the discrete norm $\vertiii{\pi u -u_h}_\ast$ in Theorem~\ref{theo: error estimate H1.1}. However, in some scenarios when the solution is rougher, let us say $u\in \Hs^{0,1}\left(Q_T\right)\cap\Hs^2\left(Q_1\cup Q_2\right)$, the estimate \eqref{eq: error estimate H1.1} turns out to be not applicable. In this case, we arrived at a nearly optimal error estimate in one spatial dimension and a sub-optimal estimate in two spatial dimensions (Theorem~\ref{theo: error estimate H1}). The nearly optimal bound $\mathcal{O}\left(h\sqrt{\abs{\log h}}\right)$ for $d = 1$ mainly relies on the Sobolev inquality \eqref{eq:Sobolev_inequality}, which is unfortunately not available for $d = 2$ (in our case, 3D space-time domain $Q_T$). 
    
    There are better bounds in the literature \cite{LR2013, Guo2021}. However, these studies relied on the interface-unfitted mesh methods, which require more effort to compute interaction integrals over the moving interface. Instead, we discretize the space-time domain using an interface-fitted mesh and avoid invoking special quadrature rules on cut elements. Moreover, the optimal error estimate can be recovered when the solution is sufficiently smooth on each subdomain.
\end{remark}

From the numerical point of view, the error $\vertiii{\pi u - u_h}_\ast$ is appropriate for studying the discrete problem \eqref{eq: discrete vf} since $\pi u$ and $\kappa_h$ are discontinuous through the discrete interface instead of the continuous one. However, it is worth noting that we also arrive at the same convergence rate when using the error $\vertiii{u-u_h}_\ast$ or $\norm{u-u_h}_{\Ys}$. For instance, we can combine \eqref{eq: discrepancy difference in u}, \eqref{eq: error estimate H1.1}, and \eqref{eq: error estimate H1} to get the following result:

\begin{corollary}
    \label{coro: Y-norm error estimate}
    Let $u\in \Xs_0$ and $u_h\in \Vs_h$ be the solutions of the variational problems \eqref{eq:vf} and \eqref{eq: discrete vf}, respectively. Assume that $u\in \Hs^{0,1}\left(Q_T\right)\cap\Hs^s\left(Q_1\cup Q_2\right)$ for given $s>\frac{d+3}{2}$. Then, the following estimate holds 
    $$
    \vertiii{u-u_h}_\ast \le C h \norm{u}_{\Hs^s\left(Q_1\cup Q_2\right)}\qqqq \text{for } d=1\text{ or }2.
    $$
    If we only have $u\in \Hs^{0,1}\left(Q_T\right)\cap\Hs^2\left(Q_1\cup Q_2\right)$, then there exists $h_0 > 0$ such that for all $h < h_0$, it holds the following
    $$
    \vertiii{u-u_h}_\ast \le C 
        \begin{Bmatrix}
            h\sqrt{\abs{\log h}}\\
            h^{\frac{2}{3}}
        \end{Bmatrix}\norm{u}_{\Hs^2\left(Q_1\cup Q_2\right)}\qqq \text{for }
        \begin{Bmatrix}
            d=1\\
            d=2
        \end{Bmatrix}.
    $$
\end{corollary}

\begin{remark}
    The case $d=3$ contains several challenges, both theoretically and computationally. From a theoretical perspective, it requires the condition $s > \frac{3+3}{2} =3$ for the embedding $\HHs^{s-1}\left(Q_T\right) \hookrightarrow \hookrightarrow \CCs\left(\ovl{Q_T}\right)$, making the assumption $u\in \Hs^{0,1}\left(Q_T\right)\cap\Hs^s\left(Q_1\cup Q_2\right)$ becomes quite strict and difficult to achieve in practice. Therefore, we will be trying to use the $\Ls^2$-projection instead of the interpolant $I_h$ to relax the smoothness condition. However, this approach requires proving a result that differs from Lemma \ref{lem: interpolation error}, which will be investigated in the future.
    
    On the other hand, from a computational standpoint, discretizing a virtual space-time domain $Q_T \subset \mathbb{R}^4$ is a complex task with high computational cost. Please refer to \cite{NK2019, DAK+2021} and the references therein. The approach for constructing a 4-dimensional virtual mesh that fits the interface could be based on the boundary-conforming method by von Danwitz et al. \cite{DAK+2021}. This would be one of the ongoing directions of our research.
\end{remark}

\section{Numerical results}
\label{sec: results}

This section is devoted to presenting various numerical examples, both in one and two spatial dimensions, to show the performance of our method. The primary focus is to confirm the convergence rate of the proposed scheme and examine its robustness in dealing with different regularity solutions and various interface shapes. We perform all numerical experiments using the finite-element based solver FreeFEM++ \cite{Hecht2012}.

\subsection{2D space-time examples}

The first two numerical experiments involve the space-time domain $Q_T = (0, 1)^2$, where the interface evolves dynamically. The initial interface \(\Gamma(0)\) comprises two points \(x_1 = 0.4\) and \(x_2 = 0.6\). The movement of the interface is governed by a smooth velocity function \(\mathrm{v}\). We define the shift function \(\mathrm{w}(t)\) as the integral of the velocity function in time
\begin{equation*}
    \mathrm{w}(t) := \int\limits_0^t \mathrm{v}(\tau) \di \tau.
\end{equation*}
The two space-time subdomains of $Q_T$ are given by
\[  
    Q_1 = \brac{(x, t) \in Q_T\mid L_1(t) < x < L_2(t), \, 0 < t < 1}, \qqqq Q_2 = Q_T \setminus \ovl{Q_1},
\]
where $L_1(t) := x_1 + \mathrm{w}(t)$ and $L_2(t) := x_2 + \mathrm{w}(t)$ are level set functions. In this section, we consider two cases (see Figure~\ref{fig:domain_2D}):
\begin{itemize}
    \item[1.] (Constant velocity) \hspace{1.06cm} $\Gamma^*$ is planar: $\mathrm{v} = 0.1$, \qqqq \hspace{0.35cm} $\mathrm{w}(t) = 0.1 t$.
    \item[2.] (Non-constant velocity) \hspace{0.4cm} $\Gamma^*$ is curved: $\mathrm{v} = 0.1 \pi \cos(2 \pi t)$, \q $\mathrm{w}(t) = 0.05 \sin (2 \pi t)$.
\end{itemize}
For verifying the convergence rate of the proposed method, the norm $\norm{\cdot}_{\Ys}$ is employed
$$
\norm{u - u_h}_{\Ys} = \paren{\int\limits_0^T\int\limits_\Om \abs{\nabla\left(u - u_h\right)}^2\dx\dt}^{1/2}.
$$
As in Corollary \ref{coro: Y-norm error estimate}, the error estimate in this norm has the same order as those presented in Theorems~\ref{theo: error estimate H1.1} and \ref{theo: error estimate H1}. The domain \(Q_T\) is discretized in both cases using a quasi-uniform mesh with \(N\) equidistant intervals along each edge. We utilize a range of discretizations of $Q_T$ with varying values of $N$ to evaluate the errors with respect to the mesh size $h$.

\begin{figure}[!htp]
\centering
    \subfigure[Example 1.]{\begin{tikzpicture}
        \draw[-latex,line width=0.2mm] (0,0)--(5,0) node[below=1mm]{$x$}; 
        \draw[-latex,line width=0.2mm] (0,0)--(0,4.7) node[left]{$t$};
        \node at (0, 0) [below=0.8mm]{$0$};
        \node at (4, 0) [below=0.6mm]{$1$};
        \node at (0,4) [left]{$1$};
        \draw[line width = 0.4mm,black] (0, 0) -- (4, 0) -- (4, 4) -- (0, 4) -- (0, 0);
        \draw[domain=0:4,smooth,variable=\y,line width = 0.3mm,black] plot ({1.6 + 0.1*\y}, {\y});
        \draw[domain=0:4,smooth,variable=\y,line width = 0.3mm,black] plot ({2.4 + 0.1*\y}, {\y});
        \node at (2.2, 2) {$Q_1$};
        \node at (1.6, -0.3) {$L_1$};
        \node at (2.4, -0.3) {$L_2$};
        \node at (1, 1) {$Q_2$};
        \node at (4.5, 2) {$Q_T$};
        \node at (3, 3) {$\Gm^{\ast}$};
        \node at (1.6, 3) {$\Gm^{\ast}$};
    \end{tikzpicture}} \qqqqq
    \subfigure[Example 2.]{\begin{tikzpicture}
        \draw[-latex,line width=0.2mm] (0,0)--(5,0) node[below=1mm]{$x$}; 
        \draw[-latex,line width=0.2mm] (0,0)--(0,4.7) node[left]{$t$};
        \node at (0, 0) [below=0.8mm]{$0$};
        \node at (4, 0) [below=0.6mm]{$1$};
        \node at (0,4) [left]{$1$};
        \draw[line width = 0.4mm,black] (0, 0) -- (4, 0) -- (4, 4) -- (0, 4) -- (0, 0);
        \draw[domain=0:4,smooth,variable=\y,line width = 0.3mm,black] plot ({1.6 + 0.2*sin(deg(2*pi*\y/4))}, {\y});
        \draw[domain=0:4,smooth,variable=\y,line width = 0.3mm,black] plot ({2.4 + 0.2*sin(deg(2*pi*\y/4))}, {\y});
        \node at (2, 2) {$Q_1$};
        \node at (1.6, -0.3) {$L_1$};
        \node at (2.4, -0.3) {$L_2$};
        \node at (1, 1) {$Q_2$};
        \node at (4.5, 2) {$Q_T$};
        \node at (2.5, 3) {$\Gm^{\ast}$};
        \node at (1.15, 3) {$\Gm^{\ast}$};
    \end{tikzpicture}}
    \caption{The space-time domain $Q_T$ with an interface evolving at the velocity $\mathrm{v} = 0.1$ in the first example (left), and $\mathrm{v} = 0.1 \pi \cos(2\pi t)$ in the second example (right).}
    \label{fig:domain_2D}
\end{figure}
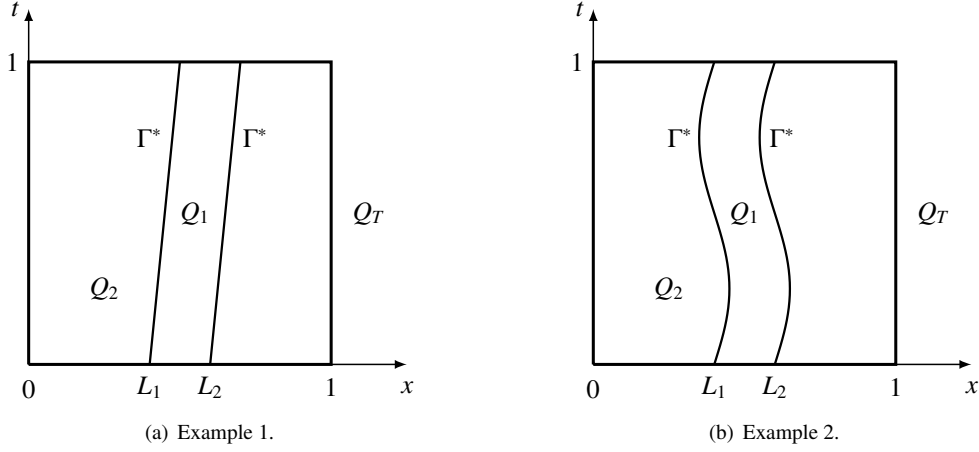

\noindent\textbf{Example 1.} In the first example, we consider the case of an interface moving with a constant velocity $\mathrm{v} = 0.1$. The level set functions are given by
\[
    L_1(t) = 0.4 + 0.1 t, \qqqq L_2(t) = 0.6 + 0.1 t.
\]
We choose the exact solution
\[
    u\left(x, t\right) = \left[\cos\paren{10\pi \left(x - L_1(t)\right)} - \cos\paren{10\pi L_1(t)}\right] \sin{\left(\dfrac{\pi t}{2}\right)},
\]
which is smooth on the whole space-time domain $Q_T$, i.e., $u \in \Hs^2\left(Q_T\right)$. It is worth noticing that this example is an exceptional case, where the solution $u$ belongs to $\Hs^2\left(Q_T\right)$ since it satisfies $\nabla u = 0$ on $\Gamma^\ast$. In general, such a regular solution does not exist owing to the jump of $\kappa$ across $\Gamma^\ast$. As suggested by classical theories, we expect that the proposed method provides an optimal convergence rate in this case. The material coefficients are set to fixed values $(\kappa_1, \kappa_2) =  (0.5,1)$, and the right-hand side of the equation \eqref{eq: IBVP} is determined accordingly. 

Table \ref{tab:example1} presents the errors and estimated convergence rates for various levels of mesh refinement. These results demonstrate a gradual improvement in errors as the mesh refinement levels increase. The results in Table \ref{tab:example1} are visualized in Figure~\ref{fig:error_2D}, offering an overall perspective on the achieved convergence rate. We see that the error convergence rate is $\mathcal{O}\left(h\right)$, which is optimal as expected. Therefore, in this scenario of a smooth solution, the performance of our method is confirmed.

\begin{table}[!htp]
\centering
\caption{The errors of the numerical method for various levels of mesh refinement in Example 1.} \label{tab:example1}
\begin{tabular}{ c c c l c } \toprule
   $N$ & dof & Mesh size $h$  & $\norm{u - u_h}_{\Ys}$ & Order  \\ \hline \hline
    10 & 3451 &  $3.410 \times 10^{-2}$             & $2.329 \times 10^0$   &  -     \\ 
    20 & 13072 & $1.897 \times 10^{-2}$         & $1.193 \times 10^{0} $ & 0.966  \\ 
    40 & 54455 &  $9.514 \times 10^{-3}$      & $5.582 \times 10^{-1}$   & 1.081   \\
    80 & 207616& $4.892 \times 10^{-3}$ & $2.948 \times 10^{-1}$ & 1.031\\
    160& 827676 & $2.525 \times 10^{-3}$ & $1.478 \times 10^{-1}$ & 1.058\\
    \bottomrule
\end{tabular}
\end{table}

\noindent\textbf{Example 2.} In the second example, a sinusoidal curved space-time interface is considered, which is determined by the velocity $\mathrm{v} = 0.1 \pi \cos{\left(2\pi t\right)}$. As a result, the level set functions have the forms
\[
    L_1(t) = 0.4 + 0.05 \sin(2\pi t), \qqqq L_2(t) = 0.6 + 0.05 \sin(2\pi t).
\] 
The same material coefficients are used, namely $\left(\kappa_1,\kappa_2\right) = \left(0.5, 1\right)$, and the exact solution is chosen as follows
\[
    u\left(x, t\right) =
    \begin{cases}
        \left[\sin\paren{20\pi \paren{x-L_1(t)} + \dfrac{\pi}{6}} + \sin \paren{10\pi L_1(t) - \dfrac{\pi}{6}}\right] \sin \paren{\dfrac{\pi t}{2}} \ & \text{for }\, \left(x, t\right) \in Q_1, \\[10pt]
        \left[\sin\paren{10\pi \paren{x-L_1(t)} + \dfrac{\pi}{6}} + \sin \paren{10\pi L_1(t) - \dfrac{\pi}{6}}\right] \sin \paren{\dfrac{\pi t}{2}} & \text{for }\, \left(x, t\right) \in Q_2.
    \end{cases}
\]
This solution is globally less regular than the one in the previous example. However, it is smooth on each space-time subdomain, i.e., $u \in \Hs^1\left(Q_T\right)\cap\Hs^3\left(Q_1\cup Q_2\right)$. Theorem~\ref{theo: error estimate H1.1} suggests that the proposed finite element method applied to this example exhibits an optimal convergence rate $\mathcal{O}\left(h\right)$. This fact is justified by the results in Table~\ref{tab:example2} and Figure~\ref{fig:error_2D}, where the estimated convergence rate is almost linear. 

While the solution has a discontinuous gradient across the moving interface, necessitating more effort, the difference between Examples 1 and 2 is insignificant when comparing the convergence rates obtained with the same mesh size. Hence, our method yields a similar level of discretization accuracy for both the case of an interface evolves with a constant velocity $\mathrm{v} = 0.1 $ and a time-dependent velocity $\mathrm{v} =0.1 \pi \cos{\left(2\pi t\right)}$. Figure~\ref{fig:solution_2D} describes the discrete solutions for both examples. To effectively simulate their different characteristics, three-dimensional visualizations are crucial. The continuity property of the solution in Example 1 is observed by the smooth curves across the interface. In contrast, the discontinuity of the solution gradient is highlighted by the fracture curves on the interface in Example 2.

\begin{table}[!htp] 
\centering
\caption{The errors of the numerical method for various levels of mesh refinement in Example 2.} \label{tab:example2}
\begin{tabular}{ c c c l c } \toprule
$N$& dof & Mesh size $h$   & $\norm{u - u_h}_{\Ys}$ & Order  \\ \hline\hline
   10 & 3370 &  $3.181 \times 10^{-2}$    & $4.561 \times 10^0$   &  -     \\ 
   20 & 13064&  $1.699 \times 10^{-2}$   & $2.338 \times 10^0 $ & 0.964  \\ 
   40 & 53074& $9.001 \times 10^{-3}$ & $1.154 \times 10^0$   & 1.038   \\
   80 & 212806& $4.302 \times 10^{-3}$ & $5.754 \times 10^{-1}$ & 1.149\\
   160& 847930& $2.459 \times 10^{-3}$  & $2.884 \times 10^{-1}$ & 1.189\\
    \bottomrule
\end{tabular}
\end{table}

\begin{figure}[!http]
\centering
    \subfigure
    {\includegraphics[scale =  .2 ]{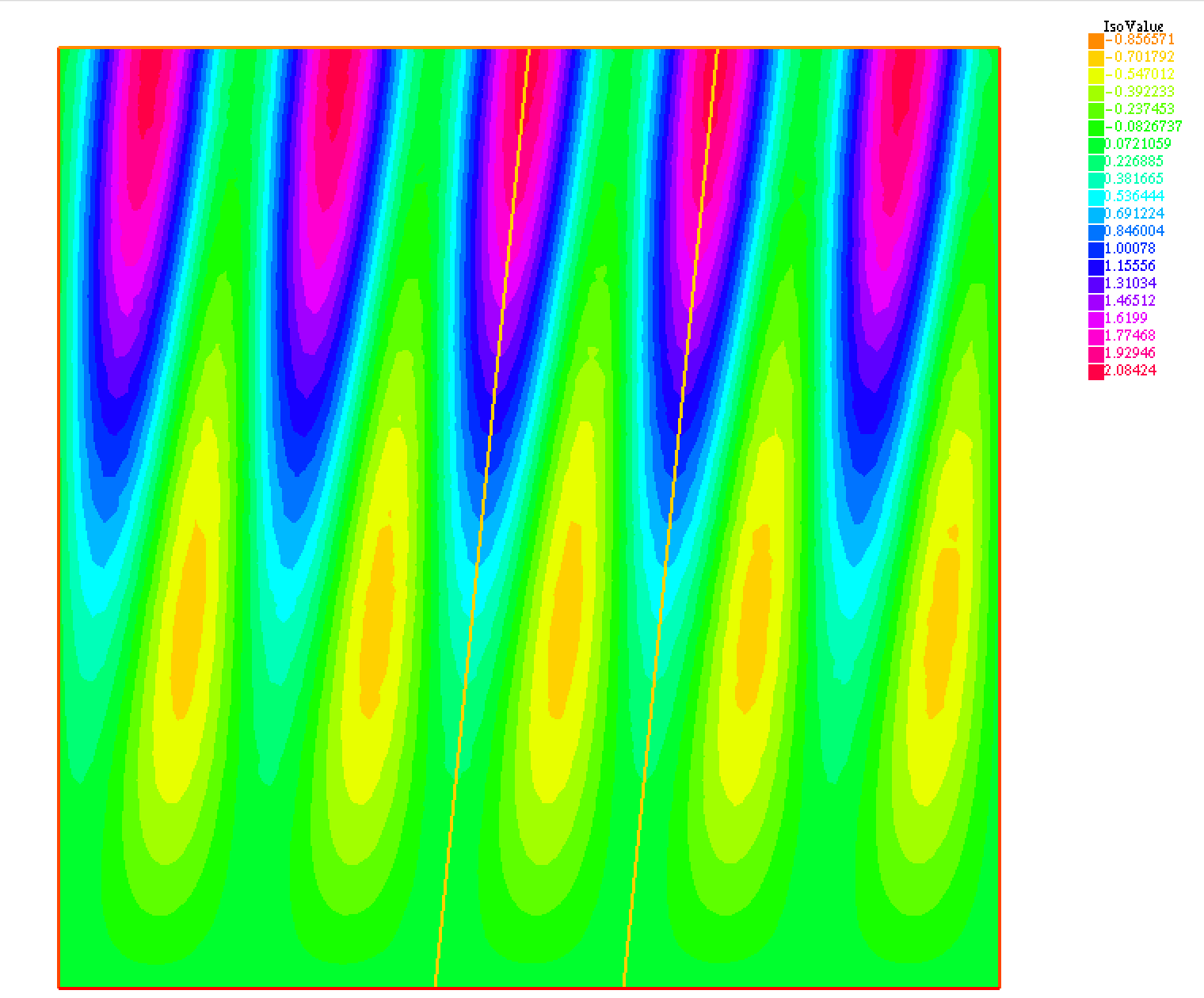}} \hspace{2.5cm}
    \subfigure{\includegraphics[scale = .2]{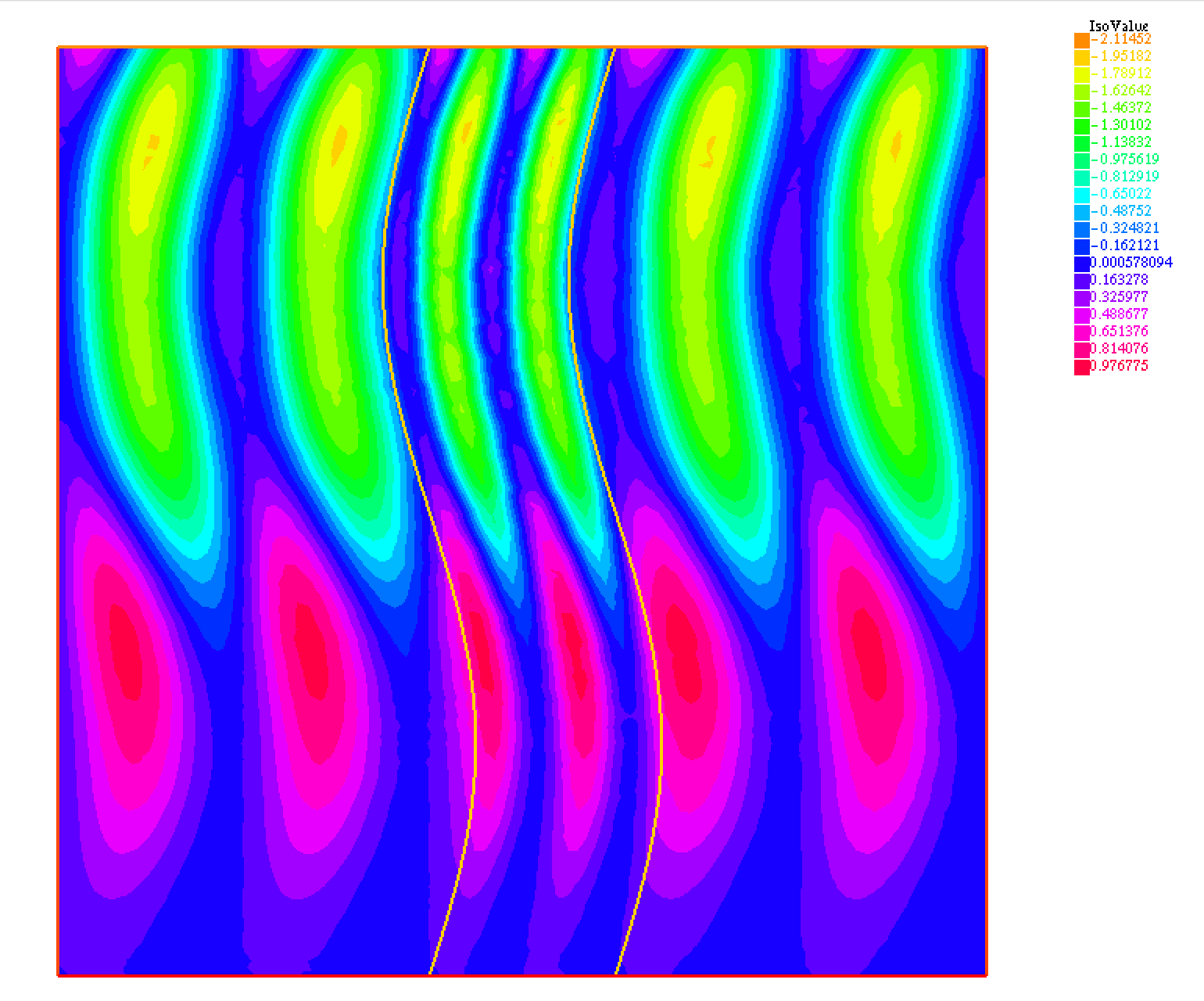}}\\
    \subfigure{\includegraphics[scale = .16]{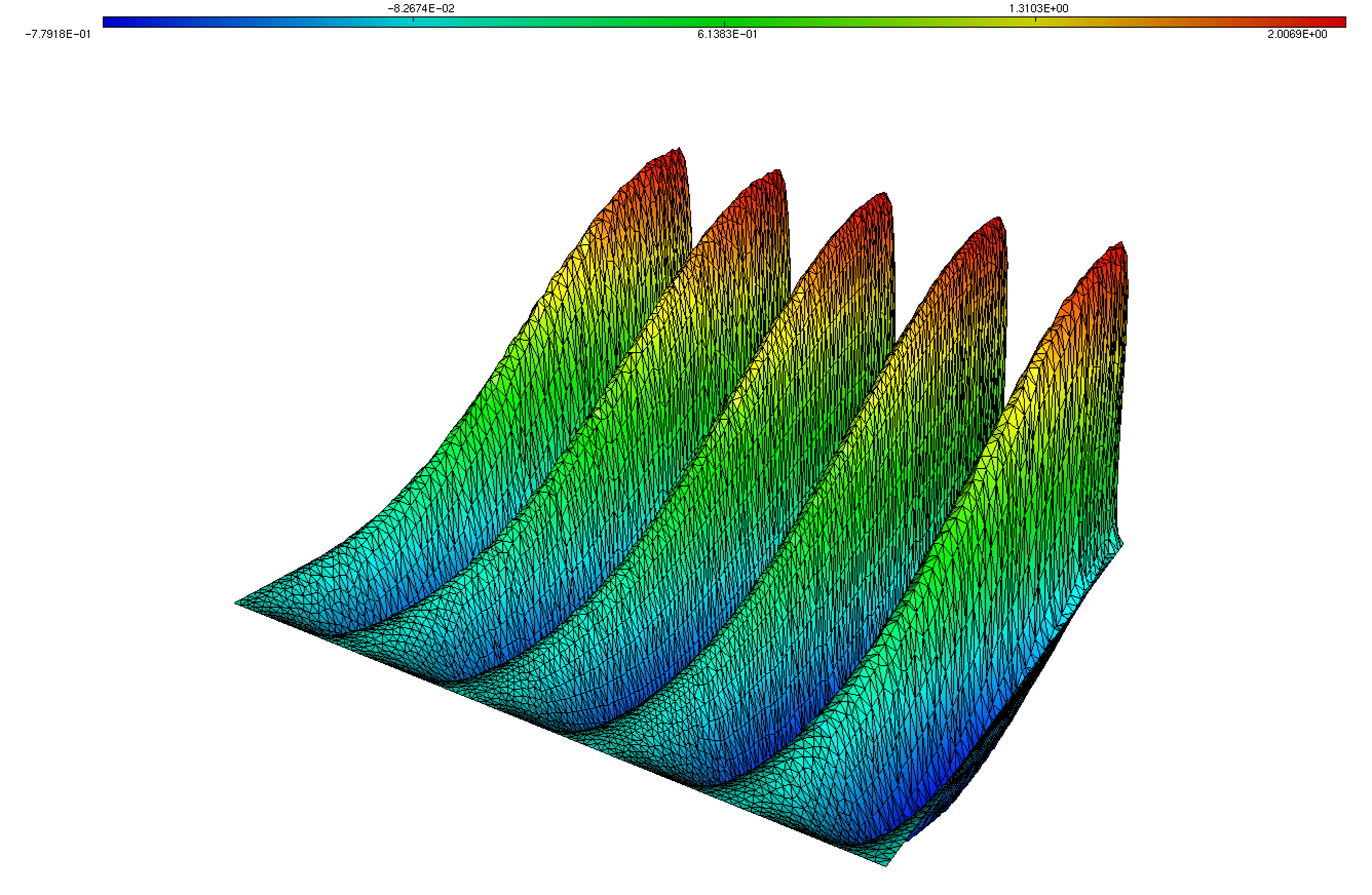}} \qq
    \subfigure{\includegraphics[scale = .1515]{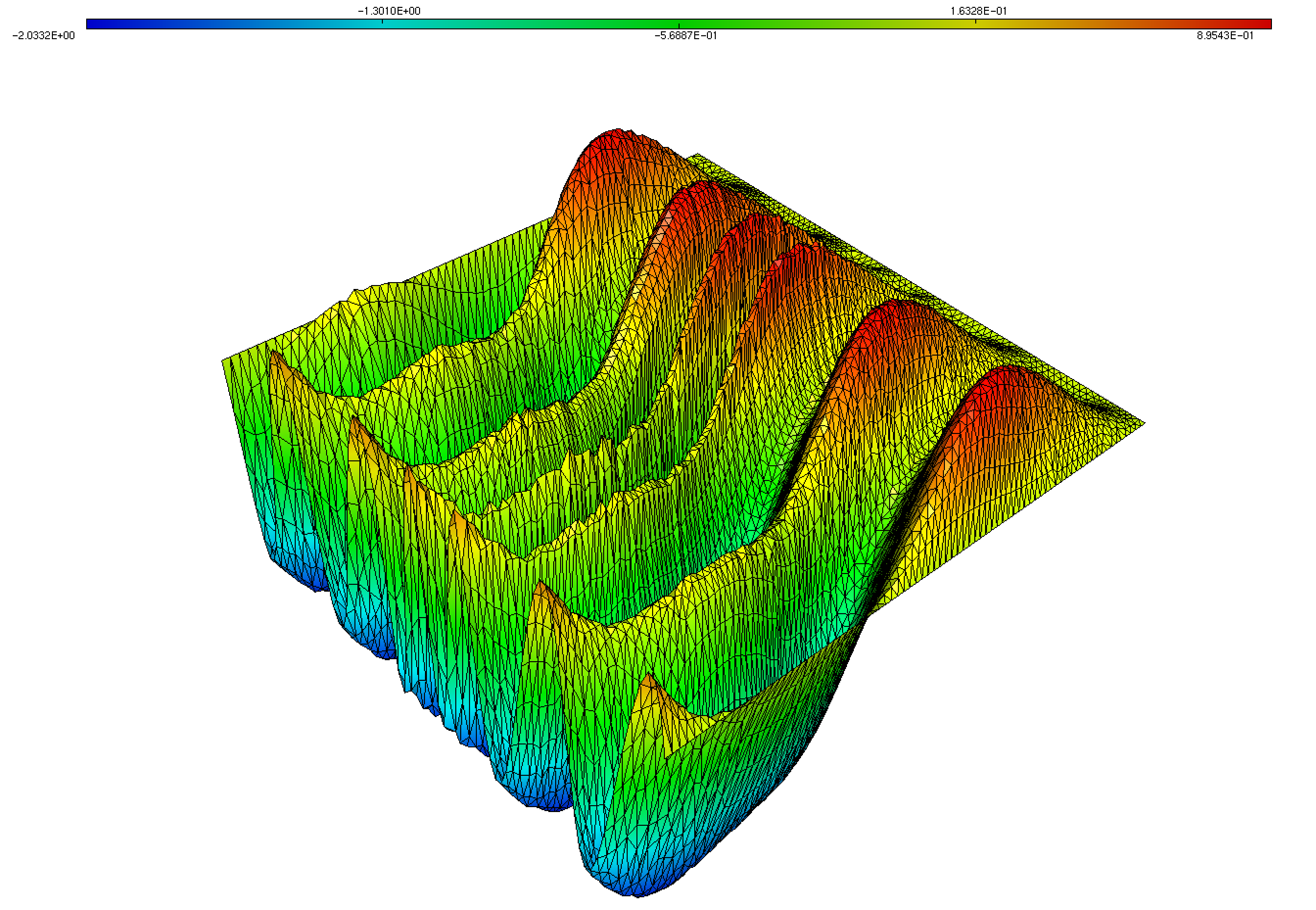}}
    \caption{The numerical solutions of Example 1 (left) and Example 2 (right).}
    \label{fig:solution_2D}
\end{figure}

\begin{figure}[!http]
\centering
    \subfigure[Example 1.]
    {\includegraphics[scale = .32 ]{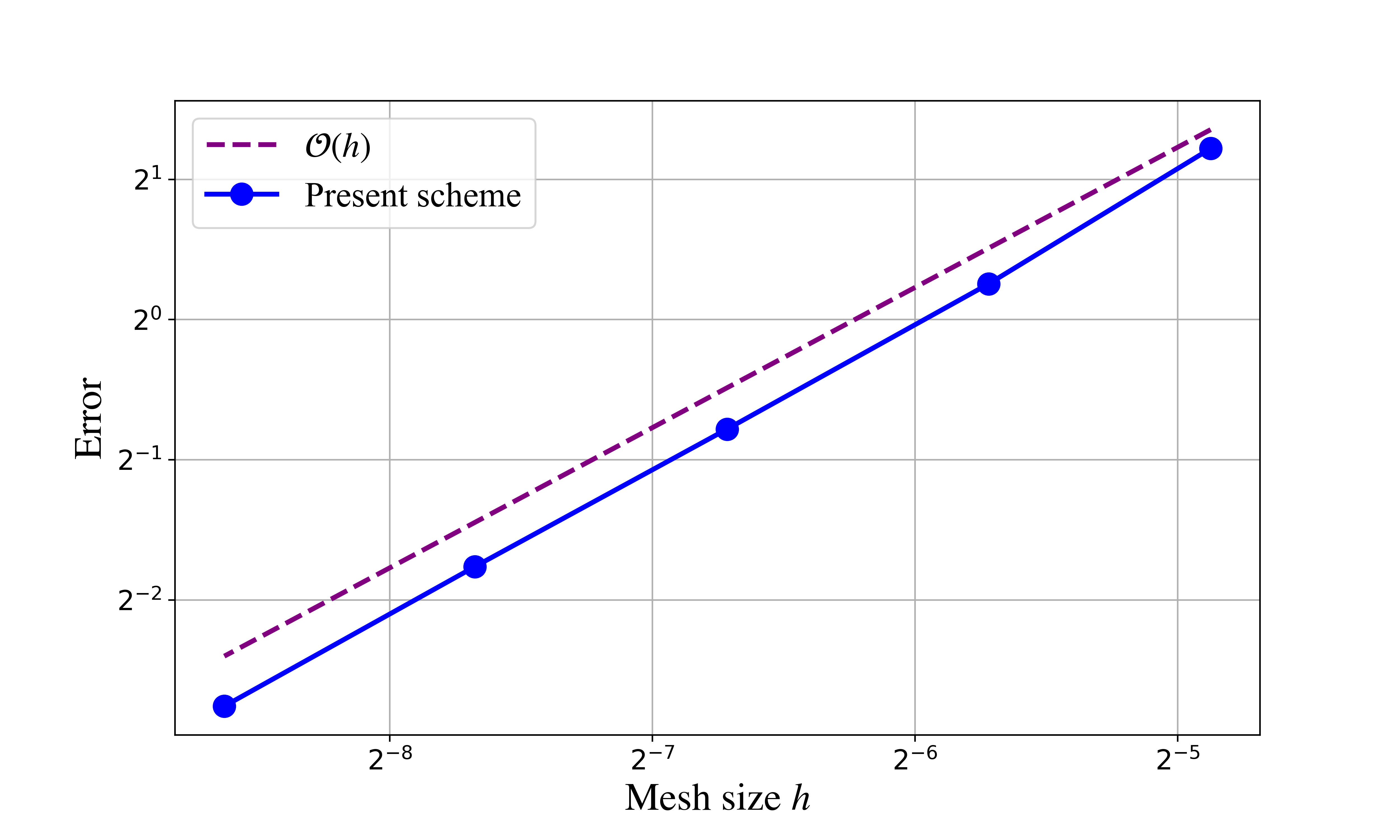}} 
    \subfigure[Example 2.]
    {\includegraphics[scale = .32 ]{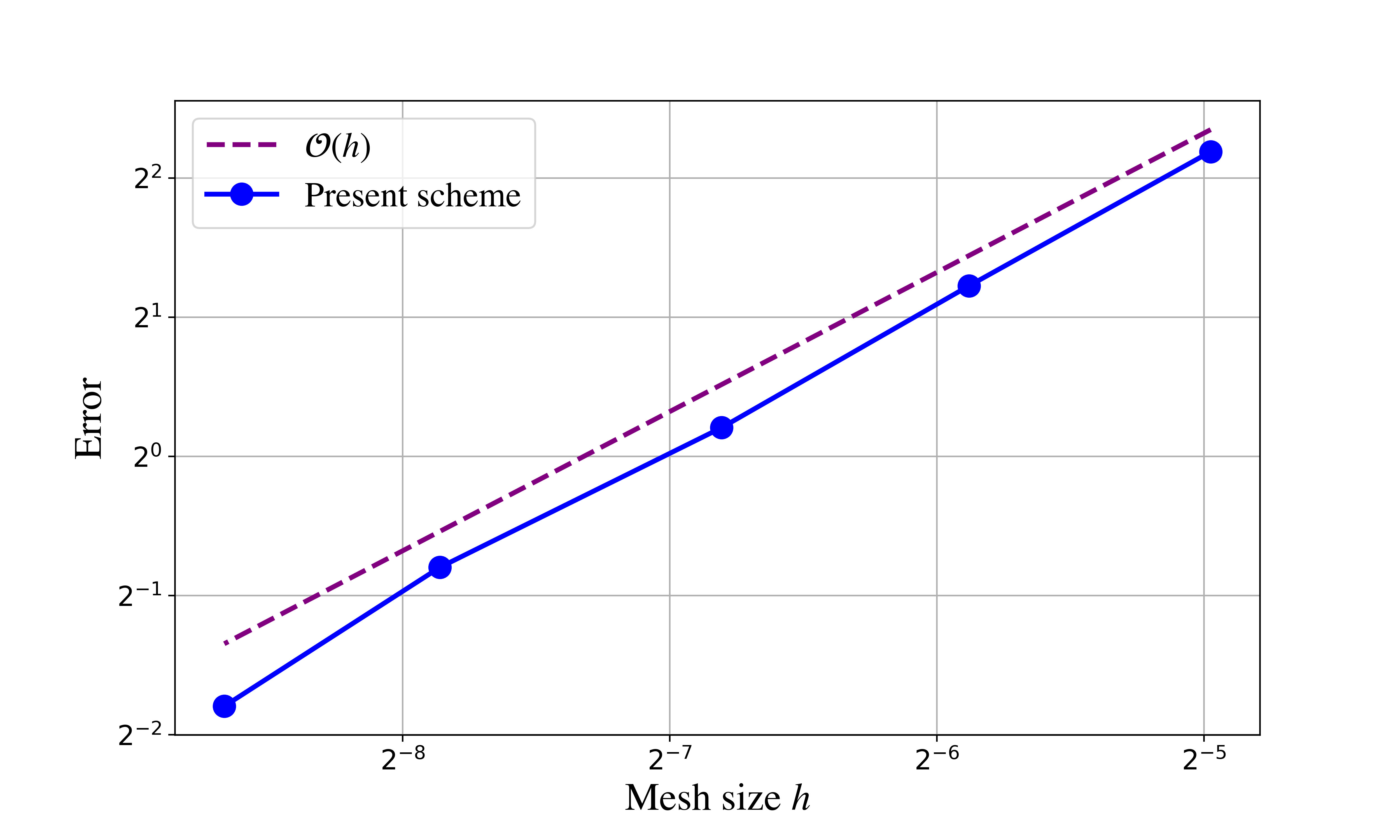}}
    \caption{The errors of the numerical scheme and the corresponding convergence rates in Examples 1 and 2.}
    \label{fig:error_2D}
\end{figure}

\subsection{3D space-time example}

\noindent\textbf{Example 3.} In this example, we will simulate an advection-diffusion problem that involves moving subdomains in the 3D space-time domain. This example is modified from an experiment performed in \cite{LSV2024} using a time-stepping scheme. Figure~\ref{fig:model_3D} illustrates the discretized space-time domain $Q_T$, where the interface undergoes a rotation around the $t$-axis over the time interval $(0, 1)$ with a velocity $\vb = (-2 \pi y, 2 \pi x)^\transpose$. The coefficient function $\kappa$ is given by $(\kappa_1, \kappa_2) = (2,1)$.

The following relative error is computed
\begin{equation*}
    \widetilde E_u = \dfrac{\norm{\nabla u_h -\nabla u_{ref}}_{\LLs^2\left(Q_T\right)}}{\norm{\nabla u_{ref}}_{\LLs^2\left(Q_T\right)}},
\end{equation*}
allowing us to assess the difference between the discrete solution $u_h$ and a reference solution $u_{ref} \in \Vs_h$. Here, $u_{ref}$ is the solution to the problem \eqref{eq: discrete vf} using a fine mesh of \( Q_T \) with 181820 tetrahedra. On the other hand, to ensure consistency with the preceding numerical examples, we will also assess the error and convergence rate in this example using the norm $\norm{u_{ref} - u_h}_{\Ys}$. Figure~\ref{fig:error_3D} portrays the convergence rates, which are gradually approaching $\mathcal{O}\left(h\right)$. These results again confirm the theoretical bound \eqref{eq: error estimate H1.1} of our proposed scheme.

\begin{figure}[htp]
\centering
    \subfigure{\begin{tikzpicture}
        \def\R{0.6}
        \def\Ri{2.5}
        \def\anga{60}
        \def\angaa{35}
        \coordinate (O) at (0,0);
        \coordinate (O1) at (1,0);
        \coordinate (R1) at (\anga:\Ri);
        \coordinate (X1) at (\angaa:{\Ri/cos(\anga-\angaa)});
        \coordinate (O2) at (-0.87,-0.5);
        \draw[fill = cyan!30] (0,0) circle (2.5cm);
        \draw[fill = black] (0,0) circle(0.02cm);
        \draw[fill = red!70, line width = 0.2mm] (O1) circle (\R);
        \draw[dashed] (O) -- (R1);
        \leftAngle{X1}{R1}{O}{0.25}
        \draw[->, line width = 0.25mm] (O) ++(\anga:\Ri) -- ++(\anga + 90: 1.5) node[midway,above] {$\vb$};
        \draw[fill = red!70] (O2) circle(\R);
        \draw[->, line width = 0.25mm] (0.7,0.9) arc(45:175:0.8) node[midway,above] {$\mathrm{\textbf{w}}(t)$};
        \node at (1,0) {$\Om_1(t)$};
        \node at (1,-1.7) {$\Om_2(t)$};
        \node at (3,-1) {$\Om$};
    \end{tikzpicture}} \qqq
    \subfigure{\includegraphics[scale = .25]{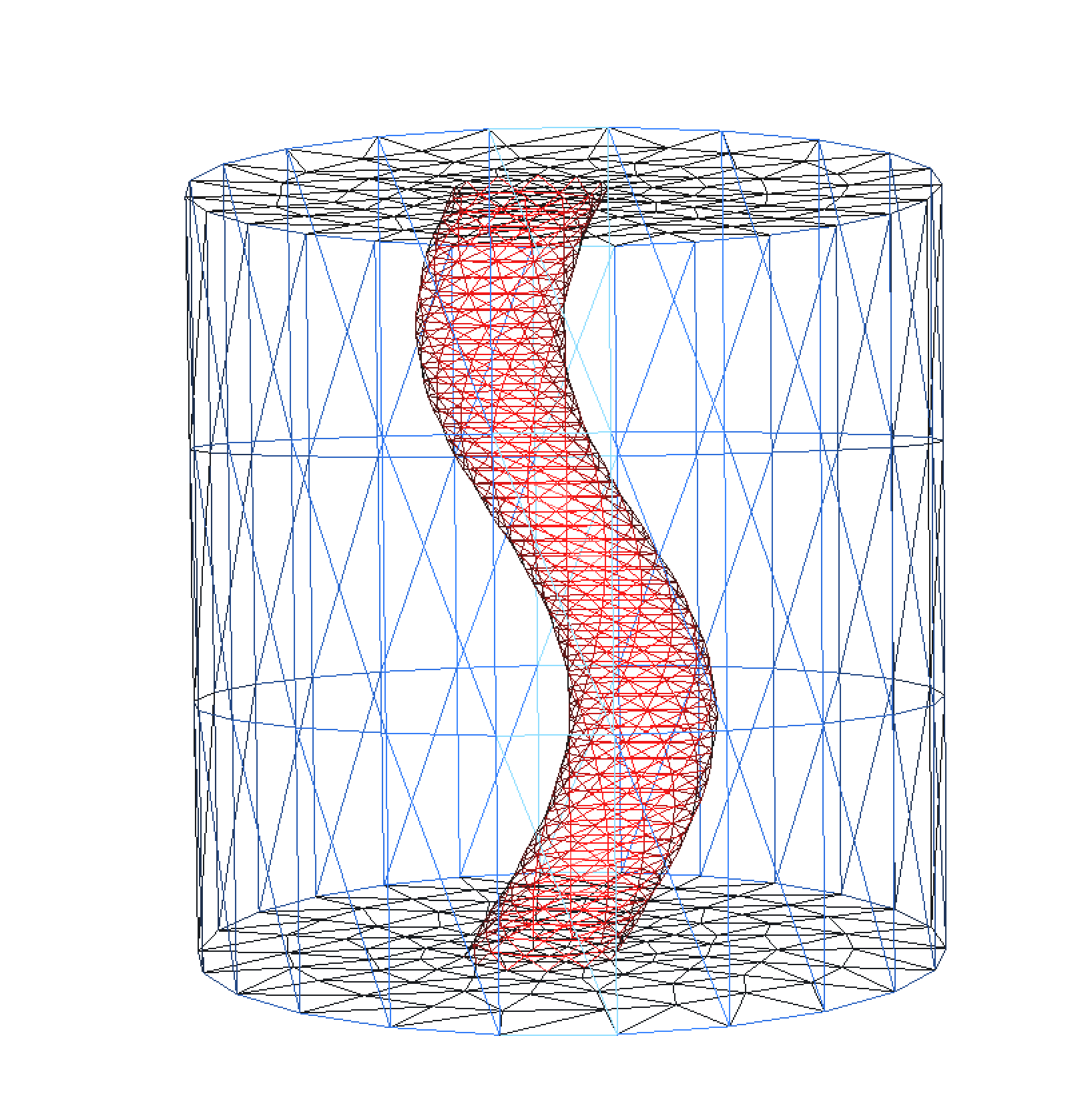}}
    \caption{The domain $\Om$ consisting of two subdomains $\Om_1(t)$ and $\Om_2(t)$  moving at velocity $\vb = (- 2 \pi y, 2 \pi x)^\transpose$ (left) and the discretized space-time domain $Q_T$ (right) in Example 3.}
    \label{fig:model_3D}
\end{figure}

\begin{figure}[!htp]
    \centering
    \subfigure[The relative error $\widetilde E_u$.]{\includegraphics[scale = .35]{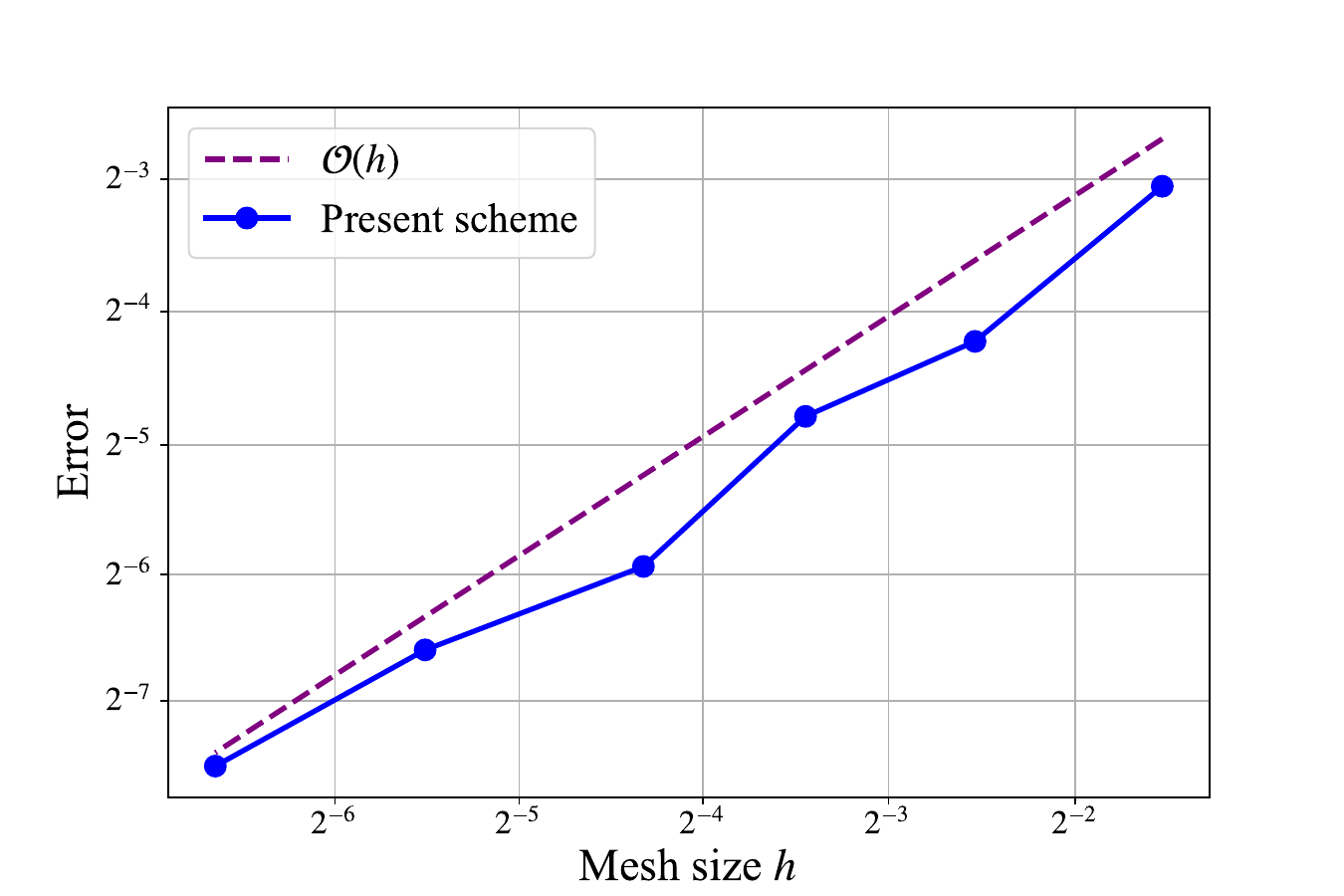}}
    \subfigure[The error $\norm{u_{ref} - u_h}_{\Ys}$.]{\includegraphics[scale = .35]{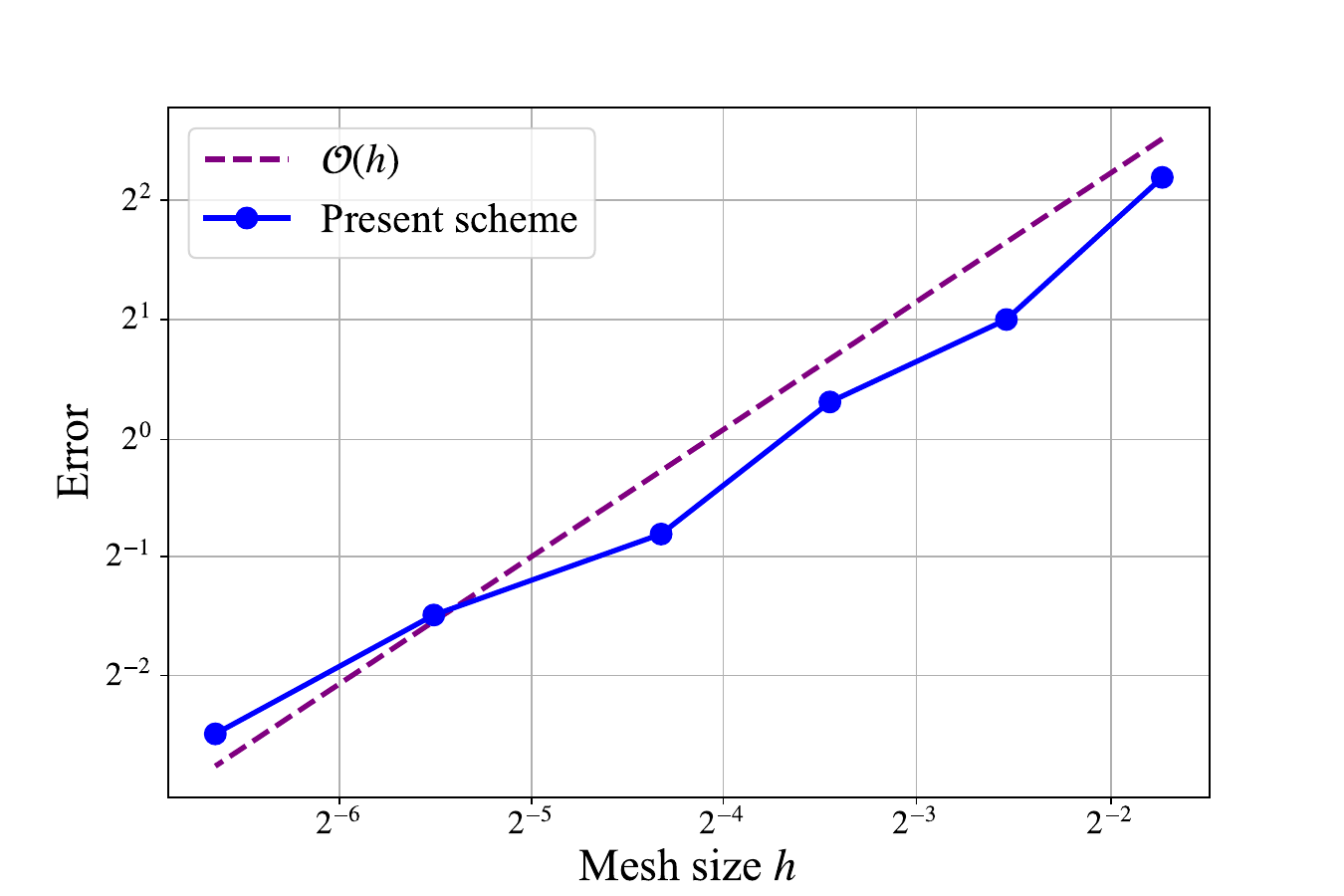}}
    \caption{The errors of the numerical scheme and the corresponding convergence rates in Example 3.}
    \label{fig:error_3D}
\end{figure}

\section{Conclusions}
\label{sec: conclusions}

Based on the interface-fitted strategy, we have proposed a new space-time finite element method to solve an advection-diffusion problem with moving subdomains. Unlike other space-time approaches, we treat the time and spatial variables similarly. We have shown an optimal error estimate with respect to a discrete energy norm under a specific globally low but locally high regularity condition. When the solution is less regular, the scheme is nearly optimal in one spatial dimension and sub-optimal in two spatial dimensions. We have also presented various numerical experiments illustrating the convergence behavior and the method's effectiveness. 

In the coming research, we intend to investigate the numerical analysis for moving-interface problems when the solution is less regular, let us say $u\in \Hs^{2,1}\left(Q_1 \cup Q_2\right)$. This situation is reasonable and more practical since the formulation does not invoke the second-order derivative of $u$ with respect to time. In addition, we would extend the proposed finite element scheme to solve moving-interfaces problems with fully discontinuous advection and diffusion coefficients, as studied in \cite{LSV2021a, LSV2024}. A more straightforward setting in \cite{LSV2021b, LSV2022c} considering a discontinuous advection coefficient but a continuous diffusion coefficient also justifies an intensive investigation using the proposed interface-fitted finite element method, which might give rise to an optimal convergence rate under lower regularity assumptions.


\bibliographystyle{elsarticle-num} 
\bibliography{abrv_ref}
\end{document}